%% file: quick_mult.tex
\documentclass{amsart}
\input{decls}
\def\fbar{{\ensuremath{\overline{f}}}}
\renewcommand{\hbar}{{\ensuremath{\overline{h}}}}
\def\fhat{\ensuremath{\widehat{f}}}
\def\ftil{\ensuremath{\widetilde{f}}}

\newcommand{\fib}[2]{{#1}^{-1}(#2)}
\newcommand{\bbZ}{\mathbb{Z}}
\newcommand{\bbQ}{\mathbb{Q}}
\newcommand{\rdual}[1]{{{#1}^{\bigstar}}}
\newcommand{\tr}{\ensuremath{\operatorname{tr}}}
\def\Ehat{\widehat{E}}

\def\fn{g}
\newcommand{\exsm}{\mathbin{\overline{\sm}}}
\CompileMatrices

\title{The multiplicativity of fixed point invariants}
\author{Kate Ponto and Michael Shulman}
\date{\today}

\thanks{The second author was supported by a National Science Foundation postdoctoral fellowship during the writing of this paper.}

\begin{document}
\maketitle

\begin{abstract}
  We prove two general factorization theorems for fixed-point invar\-iants of fibrations: one for the Lefschetz number and one for the Reidemeister trace.
  These theorems imply the familiar multiplicativity results for the Lefschetz and Nielsen numbers of a fibration.
  Moreover, the proofs of these theorems are essentially formal, taking place in the abstract context of bicategorical traces.
  This makes generalizations to other contexts straightforward.
\end{abstract}

\section{Introduction}
\label{sec:introduction}

If $p\colon E\rightarrow B$ is a fiber bundle, in which the fibers and base
are closed smooth manifolds and $B$ is connected, then it is well known that the Euler characteristics
of $E$, $B$ and any fiber $\fib{p}{b}$ are related by 
\begin{equation}
  \chi(E)=\chi(B)\cdot \chi(\fib{p}{b}).\label{eq:euler1}
\end{equation}
If $B$ is not connected, then the fibers over different components may have different Euler characteristics, so this expression must be replaced by
\begin{equation}
  \chi(E)= \sum_{C \in \pi_0(B)} \chi(C)\cdot \chi(F_C)\label{eq:euler2}
\end{equation}
where $F_C$ denotes the fiber over some point in $C$.
More generally, if $f\colon E\to E$ and $\fbar\colon B\to B$ are endomorphisms making the following square commute:
\begin{equation}\label{common}
\xymatrix{E\ar[r]^f\ar[d]_p&E\ar[d]^p\\
B\ar[r]^{\overline{f}}&B
}
\end{equation}
and moreover the fibration $p$ is \emph{orientable} (meaning that $\pi_1(B)$ acts trivially on the homology of the fiber), then a similar result holds for the Lefschetz number
\cite{hmp:pullbacks}:
\begin{equation}
  L(f) = L(\fbar) \cdot L(f_b)\label{eq:lefschetz1}
\end{equation}
(with an evident generalization to the non-connected case).
Of course, $b\in B$ must be a fixed point of \fbar\ in order for $f$ to induce an endomorphism $f_b\colon \fib{p}{b} \to \fib{p}{b}$ of the fiber over $b$.
(If \fbar\ has no fixed points, then neither does $f$, so $L(f) = L(\fbar) = 0$ and~\eqref{eq:lefschetz1} still holds vacuously.)
When $f$ and \fbar\ are identity maps, this recovers~\eqref{eq:euler1}.

Our goal in this paper is to generalize formula~\eqref{eq:lefschetz1} to the case of non-orientable fibrations, and also to the Reidemeister trace, a fixed-point invariant which refines the Lefschetz number (see below).
Note that the right side of~\eqref{eq:lefschetz1} is not even well-defined if the fibration is not orientable, since the Lefschetz number of $f_b$ can vary with $b$, even inside a single connected component of $B$.

\begin{eg}\label{eg:double-cover}
  Consider the double cover $p\colon S^1 \to S^1$, let \fbar\ be reflection in the $x$-axis (with $S^1$ considered as a subset of $\mathbb{R}^2$), and let $f$ be some map lying over it.
  Then \fbar\ has two fixed points $1$ and $-1$, and over one of them $f_b$ is the identity (with Lefschetz number $2$), while over the other it is the transposition (with Lefschetz number $0$).
\end{eg}

Since $B$ is a closed smooth manifold, $\fbar$ is homotopic to a map with a finite number of isolated fixed points.  
The  invariants considered here are invariants of homotopy classes of maps, so we may assume that $\fbar$ has only a finite number of 
fixed points.
Then a simple way to express our formula which generalizes~\eqref{eq:lefschetz1} is
\begin{equation}\label{eq:lefschetz2}
  L(f) = \sum_{\fbar(b)=b} \mathrm{ind}_b(\fbar) \cdot L(f_b).
\end{equation}
That is, we sum over the fixed points of \fbar, adding up the Lefschetz numbers of the fiberwise maps, with the indices of the \fbar-fixed points as coefficients.
For instance, in \autoref{eg:double-cover}, the fixed points of \fbar\ each have index $1$, so we compute $L(f) = 1\cdot 0 + 1\cdot 2 = 2$, which is correct.

Note that if $f$ and \fbar\ are identity maps (or, more precisely, deformations of identity maps that have isolated fixed points), then we can group the fixed points of \fbar\ according to the component of $B$ in which they lie.
Since the Euler characteristic of the fiber is constant over each component, and the sum of the indices of the fixed points of \fbar\ in a given component is the Euler characteristic of that component, in this way we recover~\eqref{eq:euler2}.
More generally, this grouping also applies whenever $p$ is orientable.

For nonorientable $p$ such a coarse grouping is not possible, as \autoref{eg:double-cover} shows, but a somewhat finer grouping is.
We say two fixed points $b_1$ and $b_2$ of \fbar\ are in the same \emph{fixed-point class} if there is a path $\gamma$ from $b_1$ to $b_2$ which is homotopic to $\fbar(\gamma)$ relative to
 endpoints.
(This is a classical definition and can be expressed in many equivalent ways; see, for example, \cite{b:lefschetz,
c:reidemeister, j:nielsen}.)
In this case, one can show that $L(f_{b_1}) = L(f_{b_2})$.

Thus, we can rewrite~\eqref{eq:lefschetz2} as
\begin{equation}\label{eq:lefschetz3}
  L(f) = \sum_{\text{fixed-point}\atop\text{classes }C} \mathrm{ind}_C(\fbar)\cdot L(f_C)
\end{equation}
where
\begin{equation}
  \mathrm{ind}_C(\fbar) = \sum_{b\in C} \mathrm{ind}_b(\fbar)\label{eq:fpc-ind}
\end{equation}
denotes the sum of the indices of all fixed-points in the class $C$, and $L(f_C)$ denotes the common value of $L(f_b)$ for any $b\in C$.

One advantage of~\eqref{eq:lefschetz3} is that unlike~\eqref{eq:lefschetz2}, it involves only homotopy-invariant quantities.
If $f$ and $\fbar$ are deformed through compatible homotopies, then while $L(f)$ remains constant, the fixed points of $\fbar$ can move around and even appear and disappear.
However, the total index~\eqref{eq:fpc-ind} of all fixed-points in a given class remains constant, as does the Lefschetz number $L(f_C)$ over any such class.

Of course, for this last statement to have any meaning, there must be a sense in which the \emph{set of fixed-point classes} remains constant as a map is deformed.
This is not literally true; \fbar\ might have two fixed points in the same class which cancel each other out under a deformation, resulting in that fixed-point class disappearing.
However, there is a set of ``potential'' fixed point classes which is homotopy invariant, defined as follows.

Let $\sh{\pi_1 B_{\fbar}}$ be the set of pairs $(b,\gamma)$ where $b$ is a point of $B$ and $\gamma$ is a path $b\leadsto \fbar(b)$, modulo the equivalence relation which sets $(b,\gamma) \sim (b',\gamma')$ if there is a path $\alpha\colon b\leadsto b'$ such that $\alpha \cdot \gamma'$ is homotopic to $\gamma \cdot \fbar(\alpha)$.
(We write $\cdot$ for path-concatenation.)
Clearly if $\fbar(b) = b$ and $c_b$ is the constant path at $b$, then $(b, c_b)$ is an element of $\sh{\pi_1 B_{\fbar}}$.
Moreover, we have $(b,c_b) \sim (b',c_{b'})$ if and only if $b$ and $b'$ are in the same fixed-point class.
Therefore, the set of fixed-point classes of \fbar\ injects into $\sh{\pi_1 B_{\fbar}}$, and the latter set is homotopy invariant; see~\cite{c:reidemeister,j:nielsen}.

If we make the obvious convention that $\mathrm{ind}_C(\fbar) = 0$ when $C$ is a potential fixed-point class containing no fixed points, then we can write~\eqref{eq:lefschetz3} as
\[
L(f) = \sum_{C \in\; \sh{\pi_1 B_{\fbar}}} \mathrm{ind}_C(\fbar)\cdot L(f_C).
\]

This leads us to another advantage of~\eqref{eq:lefschetz3} over~\eqref{eq:lefschetz2}: formula~\eqref{eq:lefschetz3} does not require us to actually find the fixed points of $\fbar$.
We only need to calculate the numbers $\mathrm{ind}_C(\fbar)$ and $L(f_C)$ for each \emph{potential} fixed-point class.
We will return to the numbers $\mathrm{ind}_C(\fbar)$ below; as for $L(f_C)$, it turns out that we can calculate it without even knowing whether $C$ contains any fixed points.
Specifically, suppose $(b,\gamma\colon b \leadsto \fbar(b))$ is a representative of $C$.
Since $p$ is a fibration, by path-lifting we can obtain a homotopy equivalence $h_\gamma\colon \fib{p}{\fbar(b)} \xto{\sim} \fib{p}{b}$.
Then $L(f_C)$ can be identified with the Lefschetz number of the composite
\[ \fib{p}{b} \xto{f_b} \fib{p}{\fbar(b)} \xto{h_\gamma} \fib{p}{b}. \]
Note that we use $f_b$ to denote the restriction of $f$ to $\fib{p}{b}$, whether or not $b$ is a fixed point of $\fbar$.

A final advantage of~\eqref{eq:lefschetz3} over~\eqref{eq:lefschetz2} is that it is more direct to prove.
To explain why, consider first the case of a trivial bundle $p\colon B\times F \to B$, where $f = \fbar \times f_b$ is also a product.
In this case, formula~\eqref{eq:lefschetz1} follows from abstract nonsense.
For Lefschetz numbers can be identified with \emph{traces} in an abstract category-theoretic sense (see \S\ref{sec:n-duality-lefschetz}), and such traces are always multiplicative:
\begin{equation}
  \tr(f\ten g) = \tr(f) \cdot \tr(g).\label{eq:trmult}
\end{equation}
In fact, the product $\tr(f) \cdot \tr(g)$ in~\eqref{eq:trmult} is properly expressed as a \emph{composite} $\tr(f) \circ \tr(g)$ of endomorphisms of a ``unit object''.
In our topological case, this ``unit object'' is a large-dimensional sphere, whose endomorphisms can be identified (up to homotopy) with integers (their degrees), and composition corresponds to multiplication.

Now formula~\eqref{eq:lefschetz3} can also be interpreted as a composite, in the following way.
Let $\bbZ\sh{\pi_1 B_{\fbar}}$ denote the free abelian group on $\sh{\pi_1 B_{\fbar}}$.
\begin{defn}
The \textbf{refined fiberwise Lefschetz number} of $f$
is the  homomorphism
\[ \widehat{L_B}(f) \colon \bbZ\sh{\pi_1 B_{\fbar}} \to \bbZ \]
which sends each basis element $C\in \sh{\pi_1 B_{\fbar}}$ to $L(f_C)$.
\end{defn}
On the other hand, we have an element
\[ R(\fbar) = \sum_{C\in \;\sh{\pi_1 B_{\fbar}}} \mathrm{ind}_C(\fbar) \cdot C
\qquad \in \bbZ\,\sh{\pi_1 B_{\fbar}}
\]
which can of course be identified with the map $\bbZ \to \bbZ\sh{\pi_1 B_{\fbar}}$ sending $1$ to $R(\fbar)$.
Formula~\eqref{eq:lefschetz3} then becomes the following.
\begin{thm}\label{thm:lefschetz-intro}
The composite
\begin{equation}
  \bbZ \xto{R(\fbar)} \bbZ\sh{\pi_1 B_{\fbar}} \xto{\widehat{L_B}(f)} \bbZ\label{eq:lefschetz4}
\end{equation}
is multiplication by $L(f)$.
\end{thm}

We will derive this by a generalization of~\eqref{eq:trmult}, using the notion of \emph{bicategorical trace} introduced in~\cite{kate:traces,PS2,PS3}.
Both $\widehat{L_B}(f)$ and $R(\fbar)$ can be identified with traces in this sense, and the analogue of~\eqref{eq:trmult} then directly implies~\eqref{eq:lefschetz3}.

Now we may turn to the question of calculating the numbers $\mathrm{ind}_C(\fbar)$, or equivalently, calculating the element $R(\fbar) \in \bbZ\sh{\pi_1 B_{\fbar}}$.
In fact, $R(\fbar)$ is a well-known invariant which refines the Lefschetz number $L(\fbar)$; it is called the \textbf{Reidemeister trace} of $\fbar$, \cite{b:lefschetz, j:nielsen, husseini:generalized}.

\begin{rmk}
  Of course, $L(\fbar)$ is the sum of all the coefficients in $R(\fbar)$, or equivalently its image under the augmentation $\bbZ\sh{\pi_1 B_{\fbar}} \to \bbZ$ that sends each basis element to $1$.
  This augmentation can be identified with the refined fiberwise Lefschetz number of the map of fibrations
  \[\vcenter{\xymatrix@-.5pc{
      B\ar[r]^{\fbar}\ar[d]_{\id_B} &
      B\ar[d]^{\id_B}\\
      B\ar[r]_{\fbar} &
      B. }}\]
  Thus, the classical relationship between $L(\fbar)$ and $R(\fbar)$ is a special case of \autoref{thm:lefschetz-intro}.
\end{rmk}

Classically, the interest of the Reidemeister trace is that, unlike the Lefschetz number, it supports a converse to the Lefschetz fixed-point theorem (under certain additional hypotheses).
\autoref{thm:lefschetz-intro} shows that the Reidemeister trace also arises naturally when we consider multiplicativity, even if we only care about Lefschetz numbers to begin with.

However, this means that in order to use \autoref{thm:lefschetz-intro} to calculate Lefschetz numbers, we must have ways to calculate the Reidemeister trace.
In particular, it would be nice to have a formula similar to~\eqref{eq:lefschetz3} for Reidemeister traces.
Conveniently, we can derive such a thing from the same abstract context that gives rise to~\eqref{eq:lefschetz3}.
The formula is what one would expect:
\begin{equation}\label{eq:reidemeister1}
  R(f) = \sum_{C \in \;\sh{\pi_1 B_{\fbar}}} \mathrm{ind}_C(\fbar)\cdot i_C(R(f_C)).
\end{equation}
where $i_C$ denotes the map
\[\bbZ\sh{\pi_1 (\fib{p}{b})_{f_b}} \to \bbZ\sh{\pi_1 E_f}
\]
(for some $b\in C$) induced by sending each fixed-point of $f_b$ to itself, regarded as a fixed point of $f$.
This map can also be calculated from a representative $(b,\gamma)$ of $C$ using a homotopy lifting, as we did for $L(f_C)$ above.

We can also express~\eqref{eq:reidemeister1} as a composite, as follows.
\begin{defn}
The \textbf{refined fiberwise Reidemeister trace} of $f$, $\widehat{R_B}(f)$, is the 
homomorphism 
\begin{equation}
  \bbZ\sh{\pi_1 B_{\fbar}} \xto{\widehat{R_B}(f)} \bbZ\sh{\pi_1 E_f}.\label{eq:reidemeister2}
\end{equation}
which sends each fixed-point class $C$ to $i_C(R(f_C))$.
\end{defn}

\begin{thm}\label{thm:reidemeister-intro}
The composite
\begin{equation}
  \bbZ \xto{R(\fbar)} \bbZ\sh{\pi_1 B_{\fbar}} \xto{\widehat{R_B}(f)} \bbZ\sh{\pi_1 E_f}.\label{eq:reidemeister3}
\end{equation}
 is $R(f)\colon \bbZ \to \bbZ\sh{\pi_1 E_f}$.
\end{thm}
We will show that these invariants can also be expressed as bicategorical traces, and so~\eqref{eq:reidemeister1} also follows from the bicategorical version of~\eqref{eq:trmult}.

In the literature, multiplicativity formulas such as~\eqref{eq:reidemeister1} have more often been expressed in terms of the \emph{Nielsen number} $N(f)$, which is the number of nonzero coefficients in the Reidemeister trace $R(f)$.
All such formulas can be derived from~\eqref{eq:reidemeister1}.
This is most easily seen using the general formulation in~\cite[Theorem 3.3]{hkw:additivity}.
(Other multiplicativity results, \cite{b:fiber,bf:fiber,hmp:pullbacks, heath, no:product, p:fiber, y:fiber}, are simplifications that follow from additional conditions
imposed on the fibration.)

To explain the relationship of~\eqref{eq:reidemeister1} to~\cite[Theorem 3.3]{hkw:additivity}, recall that for each fixed point $b$ of $\fbar$ we have a map
\[ i_b\colon \mathbb{Z}\sh{\pi_1(\fib{p}{b})_{f_b}}\rightarrow \mathbb{Z}\sh{\pi_1E_f}\]
induced by the inclusion $\fib{p}{b} \into E$.
Note that distinct fixed-point classes of $f_b$ can coalesce as fixed-point classes of $f$, so the number of nonzero coefficients of $R(f_b) \in \mathbb{Z}\sh{\pi_1(\fib{p}{b})_{f_b}}$ (that is, the Nielsen number of $f_b$) is not necessarily the same as the number of nonzero coefficients of $i_b(R(f_b))$.

However, it is true that for two fixed points $b$ and $b'$ in different fixed point classes of $\fbar$, the images $i_b(R(f_b))$ and $i_{b'}(R(f_{b'}))$
will be expressions in disjoint fixed point classes of $f$.
This is because any path exhibiting two fixed points of $f$ as being in the same fixed-point class maps down to a similar such path for their images in $B$.
This means that in~\eqref{eq:reidemeister1}:
\[R(f) = \sum_{C \in \;\sh{\pi_1 B_{\fbar}}} \mathrm{ind}_C(\fbar)\cdot i_C(R(f_C)) \]
the number of nonzero coefficients in $R(f)$ (that is, the Nielsen number of $f$) is equal to the sum, over all fixed-point classes $C$ of \fbar\ with nonzero index, of the number of nonzero coefficients in $i_b(R(f_b))$ (for any $b\in C$).

Thus, if we write $c(b)$ for the number of nonzero coefficients in $i_b(R(f_b))$, we have
\[N(f) = \sum_{i=1}^k c(b_i),\]
where $\{b_1,b_2,\ldots ,b_k\}$ is a choice of one fixed point from each fixed point class of $\fbar$ with nonzero index.
This is precisely~\cite[Theorem 3.3]{hkw:additivity}.

\begin{rmk}\label{thm:defns-of-L}
  The fixed point invariants considered here admit several equivalent definitions.
  In particular, the \emph{Lefschetz number} is commonly defined as
  \[L(f)=\sum(-1)^i\tr(H_i(f;\bbQ)).\]
  As the foregoing suggests, our methods apply directly rather to the \emph{fixed point index}, which is defined topologically in terms of fixed points.
  But since these invariants are well-known to be equal, we will not bother to distinguish between them terminologically.
  Thus, we will use the term ``Lefschetz number'' everywhere, although (as will be clear from context) the definition we actually use will be topological.
  Similarly, we will use ``Reidemeister trace'' to refer to any of the equivalent descriptions of that invariant.
\end{rmk}

Some remarks are also in order about the role of category-theoretic abstraction in this paper.
With very few exceptions, all of our results have completely formal proofs when placed in the correct abstract context.
That context is the theory of duality and trace in bicategories and indexed monoidal categories, which is described in~\cite{kate:traces,PS2,PS3}.
The presence of that abstract context is especially valuable for purposes of generalizations (such as to parametrized or equivariant theories).
In fact, we were first led to our multiplicativity formulas by way of the abstract situation.

However, in order to improve readability and minimize the necessary category-theoretic background, in this paper we have provided explicit proofs for everything instead.
The reader who is willing to learn about bicategorical traces from~\cite{kate:traces,PS2,PS3} will then find it a useful exercise to supply abstract proofs of all the theorems in this paper.
The only ingredients for our results whose proofs are not formal are Theorems~\ref{thm:mfd-cwdual} and~\ref{thm:reidemeister_id}---neither of which we prove in this paper anyway!
The former is proven in~\cite{maysig:pht} and the latter in~\cite{kate:higher}.
Both involve concrete topological arguments, supplying the ``glue'' that connects the abstract context to familiar numerical invariants.

In other words, the theorems we prove in this paper, when fully generalized, say that any invariant defined in a certain abstract way satisfies certain formal multiplicativity properties.
Therefore, once a given numerical invariant, such as the Lefschetz number or the Reidemeister trace, can be shown to have an abstract definition of this form (which is where the concrete topological work comes in), the general multiplicativity properties automatically apply to that invariant.

In \S\S\ref{sec:n-duality-lefschetz}--\ref{sec:parametrized-trace} of this paper we recall the definitions of duality and trace, specialized to the relevant examples.
We start in \S\ref{sec:n-duality-lefschetz} with the classical notions of duality and trace for topological spaces, and their relationship to Lefschetz numbers.
Then in \S\ref{sec:fib}, we recall some results about fibrations, and using these in \S\ref{sec:cw-duality} we define duality for fibrations and give examples of dualizable objects.
These definitions are due to~\cite{cw:eohc,maysig:pht}, but we describe them more simply and concretely, omitting the abstract context.
We then define traces for fibrations in \S\ref{sec:parametrized-trace} and give their connection to the Reidemeister trace.
Finally, in the last two sections we apply these ideas to prove the theorems claimed above: \autoref{thm:lefschetz-intro} in \S\ref{sec:lefschetz} and \autoref{thm:reidemeister-intro} in \S\ref{sec:reidemeister}.

\section{$n$-duality and Lefschetz numbers}
\label{sec:n-duality-lefschetz}

As background and warm-up, in this section we review the characterization of Lefschetz numbers as categorical traces, and explain how it implies the multiplicativity theorem for trivial bundles.

Let $M$ be a based topological space.
We say that $M$ is \textbf{$n$-dualizable} if there exists a based space $\rdual{M}$ and maps
\begin{align*}
  \eta &\colon S^n \too M \sm \rdual{M}\\
  \ep &\colon \rdual{M} \sm M \too S^n
\end{align*}
such that the composites
\begin{gather}
  S^n \sm M \xto{\eta\sm \id} M\sm \rdual{M}\sm M \xto{\id \sm \ep}  M \sm S^n \label{eq:tri1} \\
  \rdual{M}\sm S^n \xto{\id\sm \eta} \rdual{M}\sm M\sm\rdual{M} \xto{ \ep\sm \id} S^n \sm \rdual{M}\label{eq:tri2}
\end{gather}
become homotopic to transposition maps after smashing with some $S^m$ (in this case one says they are \emph{stably homotopic} to transpositions).
We refer to $\eta$ as the \textbf{coevaluation} and $\ep$ as the \textbf{evaluation} for the duality.

An unbased space $M$ is \textbf{$n$-dualizable} if $M$ with a disjoint basepoint, written $M_+$, is $n$-dualizable.
It is well-known (see~\cite{atiyah:thom, lms:equivariant}) that every closed smooth manifold $M$ is $n$-dualizable.
We may take $n$ to be the dimension of a Euclidean space in which $M$ embeds, and $\rdual{M}$ the Thom space of the normal bundle of the embedding.

If $M$ is an $n$-dualizable based space and $f\colon M\to M$ is a based endomorphism, we define its \textbf{trace} $\tr(f)$ to be the composite map
\begin{equation}
  S^n \xto{\eta} M \sm \rdual{M} \xto{f \sm \id} M \sm \rdual{M} \xto{\cong} \rdual{M}\sm M \xto{\ep} S^n\label{eq:tr}
\end{equation}
and its \textbf{Lefschetz number} $L(f)$ to be the degree of this trace.
The \textbf{Euler characteristic} of $M$, $\chi(M)$, is the Lefschetz number of the identity map of $M$.
We apply all these notions to unbased spaces and maps by adjoining disjoint basepoints.
These definitions are clearly homotopy invariant, and are known to agree with all other definitions of Lefschetz number and Euler characteristic; see~\cite{dp:duality}.

There are many reasons why this formulation of Lefschetz number and Euler characteristic is useful, but for us the most important is that, as suggested in \S\ref{sec:introduction}, it makes it easy to prove the multiplicativity theorem for trivial bundles.
Specifically, if $M$ and $N$ are both $n$-dualizable based spaces, as above, then it is easy and formal to prove that $M\sm N$ is also $n$-dualizable; its dual is $\rdual{N}\sm \rdual{M}$.
Moreover, if $f\colon M\to M$ and $g\colon N\to N$ are endomorphisms, we can prove by formal manipulation that $\tr(g \sm f) \sim \tr(g) \circ \tr(f)$ as maps $S^n \to S^n$, hence $L(g \sm f) = L(g) \cdot L(f)$.
Of course, if $M$, $N$, $f$, and $g$ are unbased, then $M_+ \sm N_+ \cong (M\times N)_+$ and $f_+ \sm g_+ \cong (f\times g)_+$, so we obtain the multiplicativity theorem for trivial bundles.

\begin{rmk}\label{rmk:matrix}
There is a strong analogy with the trace of a matrix.
Specifically, a finite-dimensional $\Bbbk$-vector space $V$ has a dual space $\rdual{V} = \Hom(V,\Bbbk)$, with a map $\ep\colon \rdual{V} \otimes V \to \Bbbk$ given by evaluation, and a map $\eta\colon\Bbbk \to V\otimes \rdual{V}$ defined to take $1\in\Bbbk$ to $\sum_i v_i \otimes v_i^*$, where $\{v_i\}$ is a basis of $V$ and $\{v_i^*\}$ the dual basis of $\rdual{V}$.
The analogues of~\eqref{eq:tri1} and~\eqref{eq:tri2} are identities, and the analogue of~\eqref{eq:tr} produces the map $\Bbbk \to \Bbbk$ that is multiplication by the trace of $f$ (in the classical sense).

This analogy can be made precise in the following way.
There is a general notion of duality and trace in a symmetric monoidal category; see~\cite{dp:duality, lms:equivariant, PS1}.
In the category of vector spaces, this yields the usual algebraic duality and trace, while in the stable homotopy category, it yields the notions of $n$-duality and trace defined above.
The proof of multiplicativity for trivial bundles also applies in this abstract context.
As mentioned in \S\ref{sec:introduction}, our more general multiplicativity theorems likewise apply in an abstract context of ``bicategorical'' duality and trace, which we will not make explicit; see~\cite{kate:traces,PS2,PS3} for details.
\end{rmk}

\section{Fibrations}
\label{sec:fib}

Our method to deal with nontrivial bundles will involve a refinement of $n$-duality, which is due originally to Costenoble and Waner~\cite{cw:eohc} and was studied further by May and Sigurdsson~\cite{maysig:pht}.
In order to give the definition, we need to recall some homotopy theory of fibrations.

For us, \emph{fibration} will always mean \emph{Hurewicz fibration}.
If $p_1\colon E_1 \to B$ and $p_2\colon E_2 \to B$ are fibrations, a \emph{fiberwise map} is a map $f\colon E_1\to E_2$ such that $p_2 f = p_1$.
Similarly, a \emph{fiberwise homotopy} is a map $H\colon E_1\times [0,1]\to E_2$ such that $p_2(H(e,t)) = p_1(e)$ for all $e\in E$ and $t\in [0,1]$.
We denote fiberwise homotopy equivalences by $\simeq$.

If, on the other hand, we have fibrations $p_1\colon E_1 \to B_1$ and $p_2\colon E_2 \to B_2$ over (possibly) different base spaces, and $\fbar\colon B_1 \to B_2$ is a map, then by a \emph{(fiberwise) map over \fbar} we mean a map $f\colon E_1\to E_2$ such that $p_2 f = \fbar p_1$.
Thus, a ``fiberwise map'' without qualification is always over the identity $\id_B$.

Two basic constructions on fibrations are \emph{pullback} and \emph{pushforward} along a continuous map $\fn\colon A\to B$.
Firstly, given a fibration $p\colon E\to B$, we have a pullback fibration $\fn^*E \to A$.
For any other fibration $E'\to A$, there is a natural bijection between maps $E'\to E$ over $\fn$ and fiberwise maps $E'\to \fn^*E$ (over $\id_A$).
A point of $\fn^*E$ is, strictly speaking, an ordered pair $(a,e)$ such that $\fn(a)=\fn(q(e))$, but we usually abuse notation and denote it simply by $e$; this causes no problem when we are speaking about points in the fiber over a specified point $a\in A$.

Secondly, given a fibration $q\colon E\to A$, the composite $E\xto{q} A \xto{\fn} B$ is not in general a fibration (though it is if $\fn$ is also a fibration), but up to homotopy we can replace it by one.
We denote the result by $\fn_! E \to B$.
Explicitly, a point of $\fn_! E$ over $b\in B$ is a pair $(\gamma,e)$ where $e\in E$, $\gamma$ is a path $[0,1]\to B$, and $\gamma(0) = gq(e)$ and $\gamma(1)=b$.

The pushforward operation $\fn_!$ has the property that for any fibration $E'\to B$, there is a natural bijection, up to fiberwise homotopy, between maps $E\to E'$ over $\fn$ and fiberwise maps $\fn_! E\to E'$.
Therefore, up to homotopy $\fn_!$ is left adjoint to $\fn^*$.

A particularly important construction built out of these operations is the following.
Given fibrations $M \to A\times B$ and $N\to B\times C$, we can form the product fibration $M\times N \to A\times B\times B\times C$, pull back along the diagonal $A\times B\times C \to A\times B\times B\times C$, then compose with the projection $A\times B\times C \to A\times C$ (which is a fibration).
This yields a fibration $M\times_{B} N \to A\times C$ whose total space is the pullback of $M$ and $N$ over $B$.
(We could have defined this more directly, but the description given above generalizes better to the sectioned context, below.)

This operation is associative, and it has two-sided units up to homotopy defined as follows.
For any space $B$, let $P B$ denote the space of paths $\gamma \colon [0,1]\to B$, and define the fibration $P B \to B\times B$ by evaluation at the endpoints, $\gamma \mapsto (\gamma(0),\gamma(1))$.
Then for any fibrations $M\to A\times B$ and $N\to B\times C$ we have
\[M\times_B P B \simeq M \quad\text{and}\quad P B \times_B N \simeq N.\]
Since we will be frequently using these and similar path spaces, we take the opportunity to fix notations for paths.
If $b_1$ and $b_2$ are points in $B$, we write $\beta\colon b_1 \leadsto b_2$ to indicate that $\beta$ is a path $[0,1]\to B$ with $\beta(0)=b_1$ and $\beta(1)=b_2$.
If we also have $\gamma\colon b_2 \leadsto b_3$, then $\beta\cdot\gamma$ denotes the concatenated path $b_1\leadsto b_3$.
Finally, we write $c_b$ for the constant path at $b\in B$.
(We also write $c$ for a generic point of a space $C$, but this should not cause any confusion, since the former usage always has a subscript and the latter never does.)

\begin{rmk}\label{rmk:ordering-convention}
Henceforth, when we write $M\times_{B} N$, it will always be assumed that $M$ is a fibration over $A\times B$ and $N$ is a fibration over $B\times C$, for some spaces $A$ and $C$.
It might happen that $A$ or $C$ is equal to $B$ (as for instance if $M$ or $N$ is $P B$); in such a case $M\times_{B} N$ always denotes the pullback over the middle copies of $B$.

If we are given a fibration $M\to B$, we may regard it either as a fibration $M\to B\times \star$ or as a fibration $M\to \star\times B$.
When necessary for clarity, we denote the former by $\widehat{M}$ and the latter by $\widecheck{M}$.
\end{rmk}

Next, given a map $\fn\colon A\to B$ of base spaces, we write $B_\fn$ for the pullback $(\id\times \fn)^* P B \to B\times A$, and ${}_\fn B$ for the similar pullback $(\fn\times \id)^* P B \to A\times B$.
Thus, the fiber of $B_\fn$ over $(b,a)$ is the space of paths from $b$ to $\fn(a)$, and similarly for ${}_{\fn} B$.
Strictly speaking, a point of $B_\fn$ is a pair $(\beta,a)$ such that $\beta(1)=\fn(a)$, but we usually denote such a pair abusively simply by $\beta$.
Similarly, a point of ${}_\fn B$ is a pair $(a,\beta)$ with $\beta(0)=\fn(a)$, which we usually denote by $\beta$.

The fibrations $B_\fn$ and ${}_\fn B$ are called \emph{base change objects}, and have the following important properties:
\begin{itemize}
\item For any $M\to C\times B$, we have $M\times_B B_\fn \simeq (\id\times \fn)^* M$.
\item For any $M\to B\times C$, we have ${}_\fn B \times_B M \simeq (\fn\times \id)^* M$.
\item For any $M\to C\times A$, we have $M\times_A {}_\fn B \simeq (\id\times \fn)_! M$.
\item For any $M\to A\times C$, we have $B_\fn \times_A M \simeq (\fn\times \id)_! M$.
\end{itemize}

\begin{prop}\label{thm:psfr}
  Base change objects are \emph{pseudofunctorial}.
  This means that given $A \xto{\fn_1} B \xto{\fn_2} C$, we have fiberwise homotopy equivalences
  \begin{align*}
    \left({}_{\fn_1} B\right) \times_B \left({}_{\fn_2} C\right) &\simeq {}_{\fn_2 \fn_1} C\\
    \left(C_{\fn_2}\right) \times_B \left(B_{\fn_1}\right) &\simeq C_{\fn_2 \fn_1}
  \end{align*}
  over $A\times C$ and $C\times A$, respectively.
\end{prop}
\begin{proof}
  We prove the first; the second is similar.
  A point of $\left({}_{\fn_1} B\right) \times_B \left({}_{\fn_2} C\right)$ over $(a,c)$ is a pair $(\beta,\gamma)$, where $\beta\colon \fn_1(a) \leadsto b$ and $\gamma\colon \fn_2(b)\leadsto c$ for some $b\in B$.
  On the other hand, a point of ${}_{\fn_2 \fn_1} C$ over $(a,c)$ is a path $\delta\colon \fn_2(\fn_1(a)) \leadsto c$.
  From left to right we send $(\beta,\gamma)$ to $\fn_2(\beta)\cdot \gamma$, while from right to left we send $\delta$ to $(c_{\fn_1(a)},\delta)$.
  It is easy to verify that these are inverse equivalences.
\end{proof}

\begin{prop}\label{thm:hopb}
  Suppose that
  \[\vcenter{\xymatrix{
      A \ar[r]^{\fn_1}\ar[d]_{\fn_2} \ar@{}[dr]|{\alpha} &
      B \ar[d]^{\fn_3}\\
      C \ar[r]_{\fn_4} &
      D }} \]
  is a square which commutes up to a specified homotopy $\alpha\colon \fn_4 \fn_2\sim \fn_3 \fn_1$.
  Then we have a canonical map
  \[ \left(C_{\fn_2}\right) \times_A \left( {}_{\fn_1} B\right) \too
  \left({}_{\fn_4} D \right) \times_D\left( D_{\fn_3} \right).
  \]
  If the given square is a homotopy pullback, then this map is an equivalence.
\end{prop}
\begin{proof}
  A point of $\left(C_{\fn_2}\right) \times_A \left( {}_{\fn_1} B\right)$ over $(c,b)\in C\times B$ is a pair 
  $(\gamma,\beta)$ together with an $a\in A$ such that
 $\gamma\colon c\leadsto \fn_2(a)$ and $\beta\colon \fn_1(a) \leadsto b$.
  On the other hand, a point of $\left({}_{\fn_4} D \right) \times_D\left( D_{\fn_3} \right)$ over $(c,b)$ is a pair $(\delta,\delta')$ where $\delta\colon \fn_4(c)\leadsto d$ and $\delta'\colon d\leadsto \fn_3(b)$ for some $d\in D$.
  We define the desired map by
  \begin{equation}
    (\gamma,\beta) \mapsto (\fn_4(\gamma) \cdot \alpha(a), \fn_3(\beta)).\label{eq:hopb-mapdefn}
  \end{equation}
  For the second statement, note that $\left(C_{\fn_2}\right) \times_A \left( {}_{\fn_1} B\right)$ is just the result of replacing $(\fn_2,\fn_1)\colon A\to C\times B$ by a fibration, and~\eqref{eq:hopb-mapdefn} is just the extension to this replacement of the map
  \begin{align*}
    A &\to \left({}_{\fn_4} D \right) \times_D\left( D_{\fn_3} \right)\\
    a &\mapsto (\alpha(a), c_{\fn_3(\fn_1(a))})
  \end{align*}
  over $C\times B$.
  Moreover, $\left({}_{\fn_4} D \right) \times_D\left( D_{\fn_3} \right)$ is equivalent to \emph{the} homotopy pullback of $g_3$ and $g_4$, so this map is an equivalence just when the given square is \emph{a} homotopy pullback.
\end{proof}

The fibrational version of a based space is a \emph{sectioned} fibration, i.e.\ a fibration $p\colon E\to B$ equipped with a continuous section $s\colon B\to E$ (so that $p s = \id_B$).
We usually assume that the section is a fiberwise closed cofibration, in which case we call $p$ an \textbf{ex-fibration} (see~\cite{maysig:pht} for details).
We say that a fiberwise map, or a map over $\fn$, between ex-fibrations is an \textbf{ex-map} if it also commutes with the specified sections.
Note that if $p\colon E\to B$ is an ex-fibration, then each fiber $\fib{p}{b}$ is a nondegenerately based space.

If $E\to B$ is any fibration, then $E\sqcup B \to B$ is an ex-fibration, called the result of \emph{adjoining a disjoint section} to $E$;
we denote it by $E_{+B}$.
Additionally, for any ex-fibration $E\to B$ and any (nondegenerately) based space $X$, we have an induced ex-fibration $X\sm_B E \to B$, whose fiber over $b\in B$ is the ordinary smash product $X \sm \fib{p}{b}$.

Ex-fibrations can be pulled back along a continuous map $\fn\colon A\to B$ in a straightforward way.
They can also be pushed forwards, by composing and then pushing out along the section (then replacing by a fibration, if necessary).
We have $\fn^*(E_{+B}) \simeq (\fn^*E)_{+A}$ and $\fn_!(E_{+A}) \simeq (\fn_! E)_{+B}$.

Two ex-fibrations $p\colon M\to A$ and $q\colon N\to B$ have an \textbf{external smash product} $M\exsm N \to A\times B$, whose fiber over $(a,b)$ is $\fib{p}{a}\sm \fib{q}{b}$.

Finally, given ex-fibrations $p\colon M\to A\times B$ and $q\colon N\to B\times C$, their \textbf{smash pullback} is the result of pulling back $M\exsm N$ along $A\times B\times C \to A\times B\times B\times C$, then pushing forward to $A\times C$.
This yields an ex-fibration $M\odot_B N \to A\times C$; its fiber over $(a,c)$ consists of all the smash products $\fib{p}{a,b} \sm \fib{q}{b,c}$ for all $b\in B$, with their basepoints identified (and with a suitable quotient topology which relates these products for different $b$).
We apply \autoref{rmk:ordering-convention} to smash pullbacks as well.
Moreover, we often omit the subscript on the symbol $\odot$ when there is no danger of confusion.

The smash pullback is associative, and has two-sided homotopy units given by $(P B)_{+B}$; we denote these units by $U_B \to B\times B$.
If $A$, $B$, and $C$ are all the one-point space, then ex-fibrations are just (nondegenerately) based spaces, and the smash pullback is just the smash product.
Moreover, for unsectioned fibrations $M\to A\times B$ and $N\to B\times C$, we have
\[M_{+(A\times B)} \odot_B N_{+(B\times C)} \;\simeq\; (M \times_B N)_{+(A\times C)}.\]
Thus, when speaking about maps between smash pullbacks of ex-fibrations with disjoint sections, we will usually omit to mention the points in the section; they always just ``come along for the ride''.

From now on, we change notation in the following way: for any map $\fn\colon A\to B$, we write $B_\fn$ and ${}_\fn B$ for the ex-fibrations that we would formerly have denoted $(B_\fn)_{+(B\times A)}$ and $({}_\fn B)_{+(A\times B)}$.
These ex-fibrations serve the same purpose for ex-fibrations that the unsectioned base change objects did for unsectioned fibrations.
That is, we have:
\begin{itemize}
\item For any ex-fibration $M\to C\times B$, we have $M\odot_B B_\fn \simeq (\id\times \fn)^* M$.
\item For any ex-fibration $M\to B\times C$, we have ${}_\fn B \odot_B M \simeq (\fn\times \id)^* M$.
\item For any ex-fibration $M\to C\times A$, we have $M\odot_A {}_\fn B \simeq (\id\times \fn)_! M$.
\item For any ex-fibration $M\to A\times C$, we have $B_\fn \odot_A M \simeq (\fn\times \id)_! M$.
\end{itemize}
Of course, the sectioned base change objects also satisfy Propositions~\ref{thm:psfr} and~\ref{thm:hopb}, with pullbacks replaced by smash pullbacks.

We write $S_M \to M$ for the ex-fibration $M_{+M} \simeq M\sqcup M \simeq S^0 \times M$ over $M$.
According to \autoref{rmk:ordering-convention}, we write $\widehat{S_M}$ and $\widecheck{S_M}$ for $S_M$ regarded as a fibration over $M\times \star$ and $\star\times M$, respectively.
Moreover, if $r\colon M\to \star$ is the unique map, then $\widehat{S_M}$ and $\widecheck{S_M}$ can be identified with the base change objects ${}_r \star$ and $\star_r$, respectively.
We generally prefer the less ugly notations $\widehat{S_M}$ and $\widecheck{S_M}$, but their identification with base change objects is useful, as in the following.

\begin{prop}\label{thm:sphere-bco}
  For any fibration $p\colon E\to B$, we have
  \begin{align}
    E_+ &\simeq \widecheck{S_B}\odot \widehat{E_{+B}} \label{eq:ss1}\\
    E_+ &\simeq \widecheck{E_{+B}} \odot \widehat{S_B}.\label{eq:ss2}
  \end{align}
\end{prop}
\begin{proof}
  By the identification of $\widecheck{S_B}$ as a base change object for $r\colon B\to \star$, we have
  \[ \widecheck{S_B}\odot \widehat{E_{+B}} \simeq r_! (E_{+B}) \simeq E_+ \]
  and similarly in the other case.
\end{proof}

Finally, if $M\to A\times A$ is an ex-fibration, we denote by $\sh{M}$ its \textbf{shadow}, which is a based space defined by pulling back along the diagonal $A\to A\times A$, then quotienting out the section (i.e.\ pushing forward along the unique map $A\to \star$).
For ex-fibrations $M \to A\times B$ and $N\to B\times A$, we have a canonical equivalence
\[ \sh{M \odot_B N} \;\simeq\; \sh{N\odot_A M}. \]
For a based space $X$ and an ex-fibration $M\to A\times A$, we have $\sh{X\sm_{A\times A} M} \simeq X \sm\; \sh{M}$.

Importantly, the shadow of the unit ex-fibration $U_A$ is $(\Lambda A) _+$, the free loop space of $A$ with a disjoint basepoint.
Similarly, for an endomorphism $\fn\colon A\to A$, the shadow of the base change object $A_\fn$ is $(\Lambda^\fn A)_+$, where $\Lambda^\fn A$ denotes the \emph{twisted free loop space}: its points are pairs consisting of a point $a\in A$ and a path $a\leadsto \fn(a)$.
Likewise, the shadow of ${}_{\fn} A$ is the space of paths $\fn(a)\leadsto a$ (with a disjoint basepoint), which is of course homeomorphic to $(\Lambda^\fn A)_+ \cong \sh{A_\fn}$.

Note that $\pi_0(\Lambda^\fn A)$ is the set of ``potential fixed-point classes'' denoted $\sh{\pi_1 A_\fn}$ in \S\ref{sec:introduction}.
(In particular, the 0th homology $H_0(\Lambda^\fn A)$ of the twisted free loop space, which is the 0th reduced homology of $\sh{A_\fn} \simeq (\Lambda^\fn A)_+$, is $\mathbb{Z}\sh{\pi_1 A_\fn}$.)

\begin{rmk}\label{rmk:bimod}
  Along the same lines as \autoref{rmk:matrix}, we can consider the following analogy:
  \begin{center}
    \begin{tabular}{rcl}
      spaces $A$, $B$ & $\longleftrightarrow$ & noncommutative rings $R$, $S$\\
      ex-fibrations $M\to A\times B$ & $\longleftrightarrow$ & $R$-$S$-bimodules\\
      smash pullback $\odot_B$ & $\longleftrightarrow$ & tensor product $\otimes_S$\\
    \end{tabular}
  \end{center}
  This analogy can also be made precise, using the notion of a \emph{(framed) bicategory with a shadow}.
  See~\cite{maysig:pht,kate:traces,PS2,PS3,shulman:frbi}.
  The notions of duality and trace  in the next two sections also live naturally in this abstract context.  This context is helpful 
  for perspective, but it is not necessary for the results in this paper.
\end{rmk}

\section{Costenoble-Waner duality}
\label{sec:cw-duality}

Now let $M \to A\times B$ be an ex-fibration.
We say that $M$ is \textbf{(right) $n$-dualizable} if there is an ex-fibration $\rdual{M}\to B\times A$ and fiberwise ex-maps
\begin{align*}
  \eta &\colon S^n \sm_{A\times A} U_A \too M \odot_B \rdual{M}\\
  \ep &\colon \rdual{M} \odot_A M \too S^n \sm_{B\times B} U_B
\end{align*}
such that the composites
\begin{gather*}
  S^n \sm_{A\times B} M \xto{\eta\odot \id} M\odot_B \rdual{M}\odot_A M \xto{\id \odot \ep}  M \sm_{A\times B} S^n \\
  \rdual{M}\sm_{B\times A} S^n \xto{\id\odot \eta} \rdual{M}\odot_A M\odot_B \rdual{M} \xto{ \ep\odot \id} S^n \sm_{B\times A} \rdual{M}
\end{gather*}
become fiberwise homotopy equivalent to the transposition maps after smashing with some $S^m$.
We say that an unsectioned fibration $M\to A\times B$ is $n$-dualizable if $M_{+(A\times B)}$ is so.
Obviously, if $A=B=\star$, this reduces to the original notion of $n$-duality.

We say that $M\to B$ is \textbf{fiberwise dualizable} if $\widehat{M}$ is right $n$-dualizable, and that $M$ is \textbf{Costenoble-Waner dualizable} (or \textbf{totally dualizable}) if $\widecheck{M}$ is right $n$-dualizable.

We now cite a few theorems about parametrized duality.
Most of them are purely formal, when placed in the correct abstract context (see the references).

\begin{thm}[{\cite[16.5.1]{maysig:pht}}]\label{thm:compose-duality}
  If $M\to A\times B$ and $N\to B\times C$ are right $n$-dualizable ex-fibrations, then so is $M\odot_B N \to A\times C$.
\end{thm}

\begin{thm}[{\cite[17.3.1]{maysig:pht},~\cite[5.3]{shulman:frbi}}]\label{thm:bco-dual}
  For any $\fn\colon A\to B$, the base change object ${}_{\fn} B$ is right $n$-dualizable with dual $B_{\fn}$.
\end{thm}

The evaluation and coevaluation maps $\eta$ and $\ep$ for ${}_\fn B$ can be described explicitly as follows.
We take $n=0$, so that $S^n \sm U_A \cong U_A$ and similarly for $U_B$.
The coevaluation
\[ \eta\colon U_A \to {}_\fn B \odot B_\fn \]
sends a path $\gamma\colon a \leadsto a'$ to the pair $(\fn(\gamma), c_{\fn(a')})$ (or to $(c_{\fn(a)}, \fn(\gamma))$; up to homotopy it makes no difference).
The evaluation
\[ \ep\colon B_\fn \odot {}_\fn B \to U_B \]
just concatenates two paths.
(Note that in both cases, both the domain and the codomain have a disjoint section.)

For $a\in A$, let $a$ also denote the map $a\colon \star\to A$ picking out $a\in A$.
Thus, if $p\colon M\to A$ is a fibration then $a^* M$ is an alternative notation for the fiber $\fib{p}{a}$.

\begin{thm}\label{thm:dual-char}
  An ex-fibration $M\to A\times B$ is right $n$-dualizable if and only if for each $a\in A$, the ex-fibration $(a,\id)^*M \to B$ is Costenoble-Waner dualizable.
\end{thm}
\begin{proof}
  ``Only if'' follows from Theorems~\ref{thm:compose-duality} and~\ref{thm:bco-dual}, since $(a,\id)^*M \simeq {}_a A \odot_A M$.
  We sketch a proof of ``if'' for the reader familiar with~\cite{maysig:pht}.
  By~\cite[Ch.~17]{maysig:pht}, the bicategory of parametrized spectra (in which we are implicitly working) is closed.
  Thus, by~\cite[16.4.12]{maysig:pht}, $M$ is dualizable if the map
  \[ \mu\colon M \odot (M \rhd U_B) \to M \rhd M \]
  from~\cite[(16.4.10)]{maysig:pht} is an isomorphism.
  Since the $(a,\id)^*$ jointly detect isomorphisms, it suffices for each $(a,\id)^*(\mu)$ to be an isomorphism.
  But $(a,\id)^*\simeq ({}_a A \odot -)$, so $(a,\id)^*(\mu)$ is the top morphism in the following commutative triangle:
  \[\vcenter{\xymatrix{
      {}_a A \odot M \odot (M \rhd U_B) \ar[r]\ar[dr] &
      {}_a A \odot (M \rhd M) \ar[d]\\
      & M \rhd ({}_a A \odot M). }}\]
  Now by~\cite[16.4.13(i)]{maysig:pht}, the right-hand map is an isomorphism since ${}_a A$ is dualizable, and the diagonal map is an isomorphism if $(a,\id)^*M$ is dualizable.
\end{proof}

\begin{cor}[{\cite[15.1.1]{maysig:pht}}]\label{thm:fib-dual}
  A fibration $M\to A$ is fiberwise dualizable if and only if each fiber $M_a$ is $n$-dualizable in the classical sense.
\end{cor}

The next theorem does require a significant amount of work, but fortunately it has already been done for us by~\cite{maysig:pht}.

\begin{thm}[{\cite[18.5.2]{maysig:pht}}]\label{thm:mfd-cwdual}
  If $M$ is a closed smooth manifold, then any fibration $M\to B$ is Costenoble-Waner dualizable.
  In particular, this applies to the identity fibration $M\to M$.
\end{thm}

Thus $\widecheck{S_M}$ is right $n$-dualizable if $M$ is a closed smooth manifold (in contrast to $\widehat{S_M}$, which as we have seen is always dualizable, being a base change object).

\begin{prop}\label{thm:total-dual}
  Let $p\colon E\to B$ be a fibration.  If either
  \begin{enumerate}
  \item $E\to B$ is fiberwise dualizable and $S_B$ is Costenoble-Waner dualizable, or\label{item:ss1}
  \item $E\to B$ is Costenoble-Waner dualizable,\label{item:ss2}
  \end{enumerate}
  then the space $E$ is $n$-dualizable in the classical sense.
\end{prop}
\begin{proof}
  Both conclusions follow from \autoref{thm:compose-duality} and \autoref{thm:sphere-bco}, using~\eqref{eq:ss1} for~\ref{item:ss1} and~\eqref{eq:ss2} for~\ref{item:ss2} (plus the fact that $\widehat{S_B}$ is always dualizable).
\end{proof}

In particular, if $S_B$ is Costenoble-Waner dualizable, then $B$ is $n$-dualizable.
\autoref{thm:mfd-cwdual}, and its more general version stated in~\cite{maysig:pht}, imply that the converse is usually true in practice.

For the multiplicativity of Reidemeister trace, we will need one final observation about dualizability.

\begin{prop}\label{thm:Bp-dual}
  If $p\colon E\to B$ is a fibration such that $S_{\fib{p}{b}}$ is Costenoble-Waner dualizable for each fiber $\fib{p}{b}$, then the base change object $B_p$ is dualizable.
\end{prop}
\begin{proof}
  By \autoref{thm:dual-char}, $B_p$ is dualizable if and only if for every $b\in B$, the fibration $(b,\id)^* B_p \to E$ is Costenoble-Waner dualizable.
  However, by definition of $B_p$, this is the homotopy fiber of $p$ over $b$.
  Since $p$ is a fibration, this is equivalent to the actual fiber $i_b\colon \fib{p}{b} \to E$.
  Hence the fibration $(b,\id)^* B_p \to E$ can also be written as $(i_b)_!(S_{\fib{p}{b}})$, and so
  \[ \widecheck{(b,\id)^* B_p} \simeq \left(\widecheck{S_{\fib{p}{b}}}\right) \odot_{\fib{p}{b}} \Big({}_{i_b} E\Big). \]
  But ${}_{i_b} E$ is always dualizable, while $\widecheck{S_{\fib{p}{b}}}$ is dualizable by hypothesis, so the result follows from \autoref{thm:compose-duality}.
\end{proof}

\begin{cor}\label{thm:total-cwdual}
  If $p\colon E\to B$ is a fibration such that $S_B$ is Costenoble-Waner dualizable, as is $S_{\fib{p}{b}}$ for each $b\in B$, then $S_E$ is also Costenoble-Waner dualizable.
\end{cor}
\begin{proof}
  Recall from \S\ref{sec:cw-duality} that $\widecheck{S_B}$ and $\widecheck{S_E}$ are the base change objects for the unique maps $r_B\colon B\to \star$ and $r_E\colon E\to \star$, respectively.
  Thus, by pseudofunctoriality of base change objects, the equality $r_B p = r_E$ determines an equivalence $\widecheck{S_E} \simeq \widecheck{S_B} \odot B_p$ over $E$.
  By \autoref{thm:Bp-dual}, $B_p$ is dualizable, and by assumption, so is $\widecheck{S_B}$; hence by \autoref{thm:compose-duality} so is $\widecheck{S_E}$.
\end{proof}

\section{Parametrized trace}
\label{sec:parametrized-trace}

If $M\to A\times B$ is $n$-dualizable and $f\colon M \to M$ is a fiberwise ex-map, we define its \textbf{trace} $\tr(f)$ to be the composite
\[\xymatrix{
  \sh{U_A \sm_{A\times A} S^n} \ar[d]^-{\eta} & & \sh{S^n \sm_{B\times B} U_B}
  \\
  \sh{M \odot_B \rdual{M}} \ar[r]^-{f \odot \id} &
  \sh{M \odot_B \rdual{M}} \ar[r]^{\simeq} &
  \sh{\rdual{M}\odot_A M} \ar[u]_-{\ep}
}\]
By the remarks in \S\ref{sec:fib} on shadows of units and smash products, we can regard $\tr(f)$ as a morphism
\[ S^n \sm (\Lambda A)_+ \to S^n \sm (\Lambda B)_+ \]
i.e.\ as a ``stable map'' $(\Lambda A)_+\to (\Lambda B)_+$.
In this paper, we will only care about the induced map on (reduced) $n^{\mathrm{th}}$ homology, which is equivalent to a map
\[ H_0(\Lambda A) \to H_0(\Lambda B) \]
on (unreduced) $0^{\mathrm{th}}$ homology.

\begin{eg}\label{thm:trid-bco}
  Let $B$ be any space and let $b\in B$, regarded as a map $b\colon \star \to B$.
  Then ${}_b B$ is right $n$-dualizable by \autoref{thm:bco-dual}, and so its identity map has a trace
  \[ \bbZ \cong H_0(\Lambda \star) \to H_0(\Lambda B) \cong \bbZ[\pi_0(\Lambda B)]. \]
  One can verify explicitly that this map picks out the component of the constant path at $b$.
  In \S\ref{sec:lefschetz} we will generalize this example to detect nontrivial paths in $B$.
\end{eg}

The following theorem, and its generalization in \autoref{thm:compose-traces2}, is the foundation on which our multiplicativity theorems rest.
Both are proven in~\cite{PS2} using the abstract context of bicategorical trace.

\begin{thm}[{\cite[7.5]{PS2}}]\label{thm:compose-traces1}
  Let $M\to A\times B$ and $N\to B\times C$ be $n$-dualizable, and let $f\colon M\to M$ and $g\colon N\to N$ be fiberwise ex-maps.
  Then the following triangle commutes up to stable homotopy.
  \[ \xymatrix{ S^n \sm (\Lambda A)_+ \ar[rr]^{\tr(f \odot g)} \ar[dr]_{\tr(f)} & &
    S^n \sm (\Lambda C)_+ \\
    & S^n \sm (\Lambda B)_+\ar[ur]_{\tr(g)}
    }\]
\end{thm}

The following corollary demonstrates how we can use this theorem to extract fiberwise information (though our intended applications are more complicated).

\begin{cor}\label{thm:fibtr-trivial}
  Suppose $E\to B$ is fiberwise dualizable and $f\colon E\to E$ is a fiberwise ex-map.
  Then the map induced on homology by the trace of $\widehat{f}$, a map
  \[ H_0(\Lambda(B)) \to \bbZ , \]
  sends a constant loop at $b\in B$ to the Lefschetz number of the induced map $f_b$ on the fiber over $b$.
\end{cor}
\begin{proof}
  By \autoref{thm:trid-bco}, the ex-fibration ${}_b B \to \star\times B$ is dualizable, and the trace of its identity map $\id_{{}_b B}$ is a map $\bbZ\to H_0(\Lambda B)$ which picks out the constant loop at $b$.
  Thus, the image of the constant loop at $b$ under $\tr (\fhat)$ is the image of 1 under the composite
  \[ \bbZ \xto{\tr(\id_{{}_b B})}H_0(\Lambda B) \xto{\tr(\fhat)} \bbZ. \]
  But by \autoref{thm:compose-traces1}, this is also the trace of the endomorphism $\id \odot \fhat$ of ${}_b B \odot \widehat{E_{+B}}$.
  This is a classically $n$-dualizable based space, and when we identify it with the fiber $\fib{p}{b}$, the endomorphism $\id \odot \fhat$ is identified with $f_b$.
\end{proof}

For our applications, we require a more general ``twisted'' notion of trace.
Thus, suppose as before that $M\to A\times B$ is $n$-dualizable, but suppose also that we have ex-fibrations $Q\to A\times A$ and $P\to B\times B$, and a fiberwise ex-map $f\colon Q\odot_A M \to M\odot_B P$ (over $A\times B$).
We define the \textbf{trace} of $f$ to be the composite
\[\xymatrix{
  \sh{Q \sm_{A\times A} S^n} \ar[d]^-{\eta} & & \sh{S^n \sm_{B\times B} P}
  \\
  \sh{Q \odot_A M \odot_B \rdual{M}} \ar[r]^-{f \odot \id} &
  \sh{M \odot_B P \odot_B \rdual{M}} \ar[r]^{\simeq} &
  \sh{\rdual{M}\odot_A M \odot_B P} \ar[u]_-{\ep}
}\]
Note that if $Q = U_A$ and $P=U_B$, this reduces to the previous definition of trace.
Generalizing the corresponding observation in that case, we can regard the trace as a morphism
\[ S^n \sm \sh{Q} \to S^n \sm \sh{P} \]
and we are usually interested in the induced map on homology
\[ H_0(\sh{Q}) \to H_0 (\sh{P}). \]
The generalized version of \autoref{thm:compose-traces1} is

\begin{thm}[{\cite[7.5]{PS2}}]\label{thm:compose-traces2}
  Let $M\to A\times B$ and $N\to B\times C$ be $n$-dualizable, let $Q\to A\times A$, $P\to B\times B$, and $R\to C\times C$ be ex-fibrations, and let $f\colon Q\odot_A M\to M\odot_B P$ and $g\colon P\odot_B N\to N\odot_C R$ be fiberwise ex-maps.
  Then the following triangle commutes up to stable homotopy.
  \[ \xymatrix{ S^n \sm \sh{Q} \ar[rr]^{\tr((\id_M \odot g)\circ (f\odot\id_N))} \ar[dr]_{\tr(f)} & &
    S^n \sm \sh{R} \\
    & S^n \sm \sh{P} \ar[ur]_{\tr(g)}
    }\]
\end{thm}

The most important application of twisted traces is when the twisting object(s) are base change objects.
Specifically, for any map $\fn\colon B\to B$ we have an equivalence $S_B \simeq \fn^* S_B$, which we can regard as a fiberwise ex-map
\[\widecheck{\fn}\colon \widecheck{S_B}\to \widecheck{S_B}\odot_B B_{\fn}.\]
If $S_B$ is Costenoble-Waner dualizable, then $\widecheck{\fn}$ has a trace
\[ \bbZ \to H_0(\Lambda^{\fn} B) \cong \bbZ\,\sh{\pi_1 B_{\fn}}. \]

We now identify this trace.
Suppose $B$ is a closed smooth manifold and $\fn\colon B\to B$ a continuous map with isolated fixed points.
Recall from \S\ref{sec:introduction} that two fixed points $b_1$ and $b_2$ of \fn\ are said to be in the same \emph{fixed-point class} if there is a path $\gamma$ from $b_1$ to $b_2$ which is homotopic to $\fn(\gamma)$ rel endpoints.
The set of fixed-point classes injects into the set $\pi_0(\Lambda^{\fn} B) \cong \sh{\pi_1 B_{\fn}}$ of equivalence classes of pairs $(b,\, b\leadsto \fn(b))$, and the \textbf{Reidemeister trace} of \fn\ is the formal sum
\[ R(\fn) = \sum_{C\in \;\sh{\pi_1 B_{\fn}}} \mathrm{ind}_C(\fn) \cdot C
\qquad \in \bbZ\,\sh{\pi_1 B_{\fn}} \cong H_0(\Lambda^{\fn} B)
\]
where the coefficient of each fixed-point class $C$ is the sum of the indices of all the fixed points in that class.
We can now state:

\begin{thm}[\cite{kate:higher}]\label{thm:reidemeister_id}
  Suppose that $S_B$ is Costenoble-Waner dualizable, and let $\fn\colon B\to B$ be any continuous map.
  Then the trace of the induced map $\widecheck{S_B} \to \widecheck{S_B}\odot_B B_{\fn}$ induces a map on homology:
  \[ \bbZ \cong H_0(\Lambda \star) \to H_0(\Lambda^{\fn}B) \]
  which picks out the Reidemeister trace $R(\fn)$.
\end{thm}

\begin{rmk}
  There is a difference between the Lefschetz number and the Reidemeister trace that is worth remarking on.
  Because most chain complexes that are used to compute homology consist of free abelian groups, the K\"{u}nneth theorem implies that homology takes cartesian products to tensor products.
  This means that it preserves the abstract construction of trace in \S\ref{sec:n-duality-lefschetz}, so that the Lefschetz number can be computed as a trace at the level of homology, as mentioned in \autoref{thm:defns-of-L}.
  
  By contrast, because the Reidemeister trace is a twisted trace, in its computation we have to consider modules over a ring such as $\bbZ[\pi_1(B)]$, and the chain complexes used to compute homology are not usually projective as $\bbZ[\pi_1(B)]$-modules.
  Thus, the K\"{u}nneth theorem fails, and so the Reidemeister trace cannot in general be computed as a trace at the level of homology (though it can be computed at the level of cellular chains).
  
  However, this difference is irrelevant for us, because we always construct traces at the level of topology, only applying homology afterwards to extract the numerical invariant.
\end{rmk}

There is one other general fact about traces which we will need in the following sections.
Suppose that
\[\vcenter{\xymatrix{
    A\ar[r]^{\fn}\ar[d]_{f_1} \ar@{}[dr]|{\alpha} &
    B\ar[d]^{f_2}\\
    A\ar[r]_{\fn} &
    B
  }}\]
is a homotopy commutative square, with specified homotopy $\alpha\colon \fn f_1\sim f_2 \fn$.
Then ${}_{\fn} B$ is right $n$-dualizable by \autoref{thm:bco-dual}, and by \autoref{thm:hopb} we have an induced map
\[ h\colon A_{f_1} \odot {}_{\fn} B \too {}_{\fn}B \odot B_{f_2}. \]
which therefore has a trace $\tr(h)\colon \sh{A_{f_1}} \to \sh{B_{f_2}}$.

\begin{thm}\label{thm:square-trace}
  The trace of this $h$ is homotopic to the map
  \[ \Lambda^{f_1} A_+ \to \Lambda^{f_2} B_+ \]
  which sends a path $\gamma\colon a \leadsto f_1(a)$ in $A$ to the path $\fn(\gamma) \cdot \alpha(a)$.
\end{thm}
\begin{proof}
  By definition, $\tr(h)$ is the composite
  \[ \sh{A_{f_1}} \xto{\eta}
  \sh{A_{f_1} \odot {}_{\fn} B \odot B_{\fn}} \xto{h}
  \sh{{}_{\fn} B \odot B_{f_2} \odot B_{\fn}} \xto{\simeq}
  \sh{B_{f_2} \odot B_{\fn} \odot {}_{\fn} B} \xto{\ep}
  \sh{B_{f_2}}
  \]
  By the remarks after \autoref{thm:bco-dual}, a path $\gamma\colon a\leadsto  f_1(a)$ in $\sh{A_{f_1}} \simeq \Lambda^{f_1} A_+$ goes via $\eta$ to the triple 
  $(\gamma, c_{\fn(a)}, c_{\fn(a)})$.  
  By the description of $h$ in the proof of \autoref{thm:hopb}, this triple then goes to $(\fn(\gamma)\cdot \alpha(a), c_{f_2(\fn(a))}, c_{\fn(a)})$.
  Since the remaining maps simply compose up all the paths, we obtain $\fn(\gamma) \cdot \alpha(a)$ as desired.
\end{proof}

\section{Lefschetz number}
\label{sec:lefschetz}

We can now prove \autoref{thm:lefschetz-intro} from the introduction.
For all of this section, we make the standing hypotheses that $p\colon E\to B$ is a fibration such that $S_B$ is Costenoble-Waner dualizable and each fiber $\fib{p}{b}$ is classically $n$-dualizable.
By the results cited in \S\ref{sec:n-duality-lefschetz} and \S\ref{sec:cw-duality}, it is sufficient for $B$ and each $\fib{p}{b}$ to be a closed smooth manifold.
Note that these hypotheses imply that $E\to B$ is fiberwise dualizable (by \autoref{thm:fib-dual}) and that $E$ is classically $n$-dualizable (by \autoref{thm:total-dual}).
We also suppose, as in the introduction, that $f\colon E\to E$ is a fiberwise map over $\fbar\colon B\to B$.

We will actually show a topological result that is slightly stronger than \autoref{thm:lefschetz-intro}.
In the previous section, we showed that there is a fiberwise ex-map
\[  \hbar\colon \widecheck{S_B} \to \widecheck{S_B}\odot_B B_{\fbar} \]
whose trace $\tr(\hbar)\colon S^n \to S^n \sm \Lambda^{\fbar} B_+$ induces the Reidemeister trace $R(\fbar)$ on homology.
We will show that there is a map $h$ with a trace $\tr(h)\colon S^n \sm \Lambda^{\fbar} B_+ \to S^n$ that induces the refined fiberwise Lefschetz number $\widehat{L_B}(f)$ on homology, and such that the composite
\[ S^n \xto{\tr(\hbar)} S^n \sm \Lambda^{\fbar} B_+ \xto{\tr(h)} S^n \]
has degree $L(f)$.
\autoref{thm:lefschetz-intro} then follows by the functoriality of homology.

In fact, by the adjunction between $\fbar_!$ and $\fbar^*$, the ex-map $f$ over \fbar\ can be regarded as a fiberwise ex-map $\fbar_! E_{+B} \to E_{+B}$, hence as a map
\[h\colon B_{\fbar}\odot \widehat{E_{+B}}\rightarrow \widehat{E_{+B}}.\]
A point of $B_{\fbar}\odot \widehat{E_{+B}}$ over $b\in B$ is a pair $(\gamma, e)$ where $\gamma\colon b \leadsto \fbar(p(e)) = p(f(e))$ is a path in $B$.
The map $h$ acts on such a point by transporting $f(e)$ along $\gamma$.

Now since each fiber of $p\colon E\to B$ is 
classically $n$-dualizable, by \autoref{thm:fib-dual} $\widehat{E_{+B}}$ is $n$-dualizable.
Thus, $h$ has a trace, which is a map
\[ \tr(h) \colon S^n \sm \Lambda^{\fbar} B_+ \to S^n. \]
We will show that $\tr(h)$ induces $\widehat{L_B}(f)$ on homology, and then apply \autoref{thm:compose-traces2} to identify $\tr(h) \circ \tr(\hbar)$ with $\tr(f)$.

We start with the main ingredient for this application of \autoref{thm:compose-traces2}.

\begin{lem} \label{thm:hhcomp}
The map $f\colon E_+\rightarrow E_+$ is (homotopic to) the composite
\[\xymatrix{
\widecheck{S_B}\odot\widehat{E_{+B}}\ar[r]^-{\overline{h}\odot \id}&
(\widecheck{S_B}\odot B_{\overline{f}})\odot\widehat{E_{+B}}\cong 
\widecheck{S_B}\odot (B_{\overline{f}}\odot\widehat{E_{+B}})\ar[r]^-{\id\odot h}&
\widecheck{S_B}\odot{\widehat{E_{+B}}}.
}\]
\end{lem}

\begin{proof} 
  A point of $\widecheck{S_B}\odot\widehat{E_{+B}}$ (other than the basepoint) is literally just a point of $E$, or more precisely a pair $(b,e)$ where $p(e)=b$.
  By definition, $\hbar\odot \id$ takes $(b,e)$ to to the triple $(\fbar(b), c_{\fbar(b)}, e)$.
  Then $\id\odot h$ acts on this triple by transporting $f(e)$ along the constant path $c_{\fbar(b)}$, which up to homotopy does nothing.
  Hence, the end result is $(\fbar(b), f(e)) \in \widecheck{S_B}\odot{\widehat{E_{+B}}}$, corresponding to $f(e)\in E_+$.
\end{proof}

It remains now to identify the map induced on homology by $\tr(h)$ with the refined fiberwise Lefschetz number.
Let $\gamma\in \Lambda^{\fbar} B$; thus $\gamma$ is a path $b\leadsto \fbar(b)$ in $B$.
We have an induced map
\[\xymatrix{\gamma_*\colon S^0\ar[r]&
(\Lambda^{\fbar} B)_+ \simeq \sh{B_{\overline{f}}}
}\]
which picks out $\gamma$, and we are interested in the composite 
\[\xymatrix@C=3pc{
S^n \ar[r]^-{S^n \sm \gamma_*}&
S^n \sm \sh{B_{\overline{f}}}\ar[r]^-{\tr(h)}
& S^n.
}\]
It will suffice to show that for any $\gamma$, this composite is the trace, in the sense of \S\ref{sec:n-duality-lefschetz}, of the composite map
\begin{equation}
  \fib{p}{b}\xto{f_b} \fib{p}{\fbar(b)}\xto{h_\gamma}\fib{p}{b}.\label{eq:fibidentify}
\end{equation}
As before, $h_\gamma$ denotes transport along $\gamma$ in the fibration $p$.
We will prove this using \autoref{thm:compose-traces2} again: we will exhibit a map $\omega_\gamma$ whose trace is $\gamma_*$, and such that~\eqref{eq:fibidentify} can be identified with the composite $(\id\odot h)(\omega_\gamma\odot\id)$ appearing in \autoref{thm:compose-traces2}.

Then $\gamma$ is a homotopy witnessing the homotopy-commutativity of the square
\begin{equation}\label{eq:gamma-square}
  \vcenter{\xymatrix{
      \star \ar[r]^b\ar@{=}[d] &
      B \ar[d]^{\fbar}\\
      \star \ar[r]_b &
      B.
      }}
\end{equation}
Thus, by \autoref{thm:hopb} we have an induced map
\[\xymatrix{
\omega_\gamma\colon 
{}_bB\ar[r]&
{}_bB \odot B_{\fbar}
}\]
which sends $\delta\colon b \leadsto b'$ to the pair $(\gamma,\fbar(\delta))$.

\begin{lem}\label{thm:tr-omega}
The trace of $\omega_\gamma$ is homotopic to $S^n \sm \gamma_*$.
\end{lem}
\begin{proof}
  This is just \autoref{thm:square-trace} applied to the square~\eqref{eq:gamma-square}.
\end{proof}

\begin{lem}\label{thm:fibcmpid}
  Up to homotopy, the map $\fib{p}{b}\xto{f_b} \fib{p}{\fbar(b)}\xto{h_\gamma}\fib{p}{b}$
can be identified with the composite 
\[ {}_bB \odot \Ehat \xto{\omega_\gamma\odot \id} {}_bB \odot B_{\fbar} \odot \Ehat \xto{\id \odot h} {}_bB \odot \Ehat.
\]
\end{lem}

\begin{proof}
  Suppose given $\delta\colon b\leadsto b'$, and $e\in \fib{p}{b'}$, so that $(\delta,e) \in {}_bB \odot \Ehat$.
  As defined above, $\omega_\gamma \odot \id$ sends $(\delta,e)$ to
  \[(\gamma, \fbar(\delta), e) \in {}_bB \odot B_{\fbar} \odot \Ehat.\]
  Then $\id\odot h$ acts on this by transporting $f(e)$ along the
  path $\fbar(\delta)$.

  We now have to identify ${}_bB \odot \Ehat$ with $\fib{p}{b}$ on both sides.
  The fiberwise homotopy equivalence $\fib{p}{b}\to {}_bB \odot \Ehat$ sends a point $e$ to $(c_{p(e)},e)$, while its inverse acts on a pair $(\beta,e)$ by transporting $e$ along $\beta$.
  Since transporting along $\fbar(c_{p(e)}) = c_{\fbar(p(e))}$ is (homotopic to) the identity,
  we see that the image of $e\in \fib{p}{b}$ is the result of transporting $f(e)$ along $\gamma$, which is precisely the desired composite
  \[\fib{p}{b}\xto{f_b} \fib{p}{\fbar(b)}\xto{h_\gamma}\fib{p}{b}.\qedhere\]
\end{proof}

\begin{prop}\label{thm:fibidentify}
  The map induced on homology by $\tr(h) \colon S^n \sm \sh{B_{\overline{f}}} \rightarrow  S^n$ is the refined fiberwise Lefschetz number $\widehat{L_B}(f)$.
\end{prop}
\begin{proof}
  Recall that by definition, $\widehat{L_B}(f) \colon\bbZ\sh{\pi_1 B_{\fbar}} \to \bbZ$ sends each path $\gamma\colon b \leadsto \fbar(b)$ to the Lefschetz number of the composite
  \[ \fib{p}{b} \xto{f_b} \fib{p}{\fbar(b)} \xto{h_\gamma} \fib{p}{b}. \]
  Thus, it will suffice to show that for every such $\gamma$, the action of $\tr(h)$ on the corresponding copy of $S^n$ has degree equal to this Lefschetz number, i.e.\ is homotopic to the trace of $h_\gamma \circ f_b$.
  But the inclusion of that copy of $S^n$ is just $S^n \sm \gamma_*$, which by \autoref{thm:tr-omega} is the trace of $\omega_\gamma$.
  Thus, by \autoref{thm:compose-traces2}, the composite
  \[ S^n \xto{S^n \sm \gamma_*} S^n \sm \sh{B_{\overline{f}}} \xto{\tr (h)}  S^n \]
  is homotopic to the trace of
  \[ {}_bB \odot \Ehat \xto{\omega_\gamma\odot \id} {}_bB \odot B_{\fbar} \odot \Ehat \xto{\id\odot h} {}_bB \odot \Ehat.
  \]
  But by \autoref{thm:fibcmpid}, this is the trace of $h_\gamma \circ f_b$, as desired.
\end{proof}

We can now put everything together to prove \autoref{thm:lefschetz-intro}.

\begin{thm}\label{thm:lefschetz}
  Let $p\colon E\to B$ be a fibration such that $S_B$ is Costenoble-Waner dualizable, and each fiber $\fib{p}{b}$ is $n$-dualizable in the classical sense (such as if they are closed smooth manifolds).
  Then the composite
  \begin{equation*}
    \bbZ \xto{R(\fbar)} \bbZ\sh{\pi_1 B_{\fbar}} \xto{\widehat{L_B}(f)} \bbZ
  \end{equation*}
  is multiplication by $L(f)$.
\end{thm}
\begin{proof}
  We identify $\bbZ$ with $H_n(S^n)$ and $\bbZ\;\sh{\pi_1 B_{\fbar}}$ with $H_n(S^n \sm \Lambda^{\fbar} B_+)$.
  By \autoref{thm:fibidentify}, we can then identify $\widehat{L_B}(f)$ with $H_n(\tr(h))$, and by \autoref{thm:reidemeister_id} we can identify $R(\fbar)$ with $H_n(\tr(\hbar))$.
  Thus, by the functoriality of homology, it suffices to prove that the composite
  \[ S^n \xto{\tr(\hbar)} S^n \sm \Lambda^{\fbar} B_+ \xto{\tr(h)} S^n \]
  is homotopic to $\tr(f)$.
  But this follows from \autoref{thm:compose-traces2} and \autoref{thm:hhcomp}.
\end{proof}

\section{Reidemeister trace}
\label{sec:reidemeister}

We now move on to \autoref{thm:reidemeister-intro} from the introduction.
For this we must strengthen our standing assumptions from the previous section to include that $S_{\fib{p}{b}}$, as well as $S_B$, is Costenoble-Waner dualizable for each $b\in B$.
By \autoref{thm:total-dual}, this implies our previous assumption that $\fib{p}{b}$ is classically $n$-dualizable, but by \autoref{thm:mfd-cwdual} it still suffices to assume that $B$ and each fiber are closed smooth manifolds.
Note that by \autoref{thm:total-cwdual}, our stronger hypotheses imply that $S_E$ is also Costenoble-Waner dualizable.

Also as in the last section, we prove a topological result slightly stronger than \autoref{thm:reidemeister-intro}, and in outline the proof is quite similar.
Consider $f\colon E\rightarrow E$ simply as an endomorphism of the space $E$ (not as a map of fibrations).
Since $S_E$ is Costenoble-Waner dualizable, by \autoref{thm:reidemeister_id} we have a fiberwise map
\[\ftil\colon \widecheck{S_E}\rightarrow \widecheck{S_E}\odot E_f.\]
over $\star\times E$, whose trace induces $R(f)$ on homology.

Recall the map
\[ \hbar\colon \widecheck{S_B} \to \widecheck{S_B}\odot_B B_{\fbar} \]
whose trace $\tr(\hbar)\colon S^n \to S^n \sm \Lambda^{\fbar} B_+$ induces the Reidemeister trace $R(\fbar)$ on homology.
We will now define a map having a trace of the form
\[S^n \sm \Lambda^{\fbar} B_+ \to S^n \sm \Lambda^f E_+,\]
which induces the refined fiberwise Reidemeister trace $\widehat{R}_B(f)$ on homology.
We will then apply \autoref{thm:compose-traces2}, as before.

Now, by pseudofunctoriality of base change objects (\autoref{thm:psfr}), the equality $\fbar p = p f$ determines a fiberwise equivalence
\[ h \colon B_{\fbar} \odot B_p \xto{\simeq} B_p \odot E_f\]
over $B\times E$.
This will play the role of the $h$ from the previous section.
A point of $B_{\fbar} \odot B_p$ over $(b,e)$ is a pair $({\alpha},{\beta})$, where $\alpha\colon b\leadsto \fbar(b')$ and $\beta\colon b'\leadsto p(e)$ for some $b'\in B$.
Similarly, a point of $B_p \odot E_f$ over $(b,e)$ is a pair $({\gamma},{\delta})$, where $\gamma\colon b\leadsto p(e')$ and $\delta\colon e' \leadsto f(e)$ for some $e'\in E$.
The map $h$ sends $(\alpha,\beta)$ to $(\alpha\cdot\fbar(\beta), c_{f(e)})$; its homotopy inverse sends $(\gamma,\delta)$ to $(\gamma \cdot p(\delta), c_{p(e)})$.

We also have an equivalence
\begin{equation}
  \widecheck{S_E} \simeq \widecheck{S_B} \odot B_p\label{eq:SEequiv}
\end{equation}
determined by the equality $r_B p = r_E$, as in the proof of \autoref{thm:total-cwdual}.
In one direction, this equivalence sends $e$ to $(p(e),c_{p(e)})$.
In the other direction it sends a pair $(b,\beta)$, where $\beta\colon b\leadsto p(e)$, to the point $e$.

We begin with the analogue of \autoref{thm:hhcomp}, whose proof is also analogous.

\begin{lem}\label{thm:hhcomp2}
Modulo the equivalence~\eqref{eq:SEequiv}, the fiberwise map $\ftil\colon \widecheck{S_E}\rightarrow \widecheck{S_E}\odot E_f$
is (homotopic to) the composite
\begin{equation}
\xymatrix{
\widecheck{S_B}\odot B_p\ar[r]^-{\overline{h}\odot \id}&
\widecheck{S_B}\odot B_{\overline{f}}\odot B_p\ar[r]^-{\id\odot h}&
\widecheck{S_B}\odot B_p\odot E_f.
}\label{eq:hhcomp2}
\end{equation}
\end{lem}

\begin{proof}
  By definition, $\ftil$ sends a point $e$ to the pair $(f(e), c_{f(e)})$.
  Now according to the above description of~\eqref{eq:SEequiv}, we should inspect the action of~\eqref{eq:hhcomp2} on the pair $(p(e),c_{p(e)})$.
  By definition of \hbar, in $\widecheck{S_B}\odot B_{\overline{f}}\odot B_p$ we end up with the triple $(\fbar(p(e)),c_{\fbar(p(e))}, c_{p(e)})$.
  But by definition of $h$, this maps to $(p(f(e)),c_{p(f(e))}, c_{f(e)})$ in $\widecheck{S_B}\odot B_p\odot E_f$.
  The inverse of~\eqref{eq:SEequiv} then forgets the path $c_{p(f(e))}$, yielding the pair $(f(e),c_{f(e)})$ as desired.
\end{proof}

By \autoref{thm:Bp-dual}, $B_p$ is $n$-dualizable.
Therefore, $h$ has a trace
\[\tr(h) \colon \sh{B_{\fbar}} \to \sh{E_f} \]
which, on homology, induces a map
\[\bbZ\sh{\pi_1 B_{\fbar}} \to \bbZ\sh{\pi_1 E_f}.\]
We want to identify this with the refined fiberwise Reidemeister trace $\widehat{R_B}(f)$.
As in the previous section, we do this pointwise for all $\gamma \in \Lambda^{\fbar} B$, by composing with the same map $\omega_\gamma$ defined there.
We can reuse \autoref{thm:tr-omega} exactly, but we need a refined version of \autoref{thm:fibcmpid}.

Note first that since $p$ is a fibration, the square
\[\vcenter{\xymatrix{
    \fib{p}{b}\ar[r]^-{i_b}\ar[d]_r &
    E \ar[d]^p\\
    \ast \ar[r]_b &
    B
  }} \]
is a homotopy pullback.
Therefore, by \autoref{thm:hopb} we have an equivalence
\[\widecheck{S_{\fib{p}{b}}}\odot {}_{i_b} E \xto{\simeq} {}_bB \odot B_p. \]
A point of $\widecheck{S_{\fib{p}{b}}}\odot {}_{i_b} E$ over $e\in E$ is a pair $(e',\delta)$, where $e'\in \fib{p}{b}$ and $\delta\colon e'\leadsto e$ is a path in $E$; this equivalence sends such a point to the pair $(c_b, p(\delta))$.

Let $k_b$ denote the composite map
\[\fib{p}{b}\xto{f_b} \fib{p}{\fbar(b)}\xto{h_\gamma}\fib{p}{b}.\]
which induces a map
\[\widetilde{k_b}\colon \widecheck{S_{\fib{p}{b}}}\rightarrow \widecheck{S_{\fib{p}{b}}}\odot (\fib{p}{b})_{k_b}\]
over $\star\times \fib{p}{b}$.
By definition, $\widetilde{k_b}$ sends each $e\in \fib{p}{b}$ to the pair $(k_b(e),c_{k_b(e)})$.
Since $S_{\fib{p}{b}}$ is Costenoble-Waner dualizable, $\widetilde{k_b}$ has a trace, which by \autoref{thm:reidemeister_id} calculates the Reidemeister trace of $k_b$.

Now the path-lifting operation which defines $h_\gamma$ actually gives us a homotopy $\theta$ in the lower square below (the upper square commutes by definition):
\begin{equation}\label{eq:ell-square}
\vcenter{\xymatrix{
    \fib{p}{b}\ar[d]_{f_b} \ar[r]^-{i_b} \ar@/_3pc/[dd]_{k_b} &
    E \ar[d]^f \\
    \fib{p}{\fbar(b)}\ar[r]^-{i_{\fbar (b)}}\ar[d]_{h_\gamma} \ar@{}[dr]|{\theta} &
    E\ar@{=}[d]^{\id_E}\\
    \fib{p}{b} \ar[r]_-{i_b} &
    E}}
\end{equation}

Thus, by \autoref{thm:hopb} we have a map
\[ \ell\colon (\fib{p}{b})_{k_b} \odot {}_{i_b} E \too {}_{i_b} E  \odot E_f. \]

\begin{lem}\label{thm:fibcmpid2}
  The following diagram commutes up to homotopy:
  \[
  \xymatrix@C=3pc{
  \widecheck{S_{\fib{p}{b}}}\odot {}_{i_b} E \ar[r]^-{\widetilde{k_b} \odot \id} \ar[d]_\simeq &
  \widecheck{S_{\fib{p}{b}}} \odot (\fib{p}{b})_{k_b} \odot {}_{i_b} E \ar[r]^-{\id \odot \ell} &
  \widecheck{S_{\fib{p}{b}}} \odot {}_{i_b} E  \odot E_f  \ar[d]^\simeq\\
  {}_bB \odot B_p \ar[r]_-{\omega_\gamma\odot \id}  &
  {}_bB \odot B_{\fbar} \odot B_p \ar[r]_-{\id \odot h} &
  {}_bB \odot B_p \odot E_f
  }\]
\end{lem}
\begin{proof}
  As remarked above, a point of $\widecheck{S_{\fib{p}{b}}}\odot {}_{i_b} E$ over $e\in E$ is a pair $(e',\delta)$ where $e'\in \fib{p}{b}$ and $\delta\colon e' \leadsto e$.
  The upper-right composite sends this to $(k_b(e'),c_{k_b(e')}, \delta)$, then to $(k_b(e'),\theta(e'), f(\delta))$, and then to $(c_b, p (\theta(e')), f(\delta)) = (c_b, \gamma, f(\delta))$.
  On the other hand, the left-bottom composite sends $(e',\delta)$ to $(c_b,p(\delta))$, then to $(\gamma, c_{\fbar(b')}, p(\delta))$, then to $(\gamma, \fbar(p(\delta)),c_{f(e)})$.
  But we can deform $(c_b, \gamma, f(\delta))$ along $\gamma$ to obtain $(\gamma, c_{\fbar(b)}, f(\delta))$, and then along $f(\delta)$ to obtain $(\gamma, p(f(\delta)), c_{f(e)}) = (\gamma, \fbar(p(\delta)), c_{f(e)})$ as desired.
\end{proof}

Now since ${}_{i_b} E$ is $n$-dualizable by \autoref{thm:bco-dual}, $\ell$ has a trace
\[ \tr (\ell) \colon \sh{(\fib{p}{b})_{k_b}} \too \sh{E_f}. \]

\begin{lem}\label{thm:inclusion}
  On homology the trace of $\ell$ induces  the map
  \[ H_0(\Lambda^{k_b} (\fib{p}{b})) \cong \bbZ\sh{\pi_1 (\fib{p}{b})_{k_b}} \too
  \bbZ\sh{\pi_1 E_f} \cong H_0 (\Lambda^f E)
  \]
  from the introduction, obtained by regarding a path in $\fib{p}{b}$ as a path in $E$.
\end{lem}
\begin{proof}
  This is just \autoref{thm:square-trace} applied to the square~\eqref{eq:ell-square}.
\end{proof}

\begin{prop}\label{thm:fibidentify2}
  The map induced on homology by $\tr(h) \colon S^n \sm \sh{B_{\overline{f}}} \rightarrow  S^n \sm \sh{E_f}$ is the refined fiberwise Reidemeister trace $\widehat{R_B}(f)$.
\end{prop}
\begin{proof}
  Recall that by definition, $\widehat{R_B}(f) \colon\bbZ\sh{\pi_1 B_{\fbar}} \to \bbZ\sh{\pi_1 E_f}$ sends each path $\gamma\colon b \leadsto \fbar(b)$ to the image under $i_b$ of the Reidemeister trace of the composite
  \[ \fib{p}{b} \xto{f_b} \fib{p}{\fbar(b)} \xto{h_\gamma} \fib{p}{b} \]
  (which we have called $k_b$).
  Thus, it will suffice to show that for every such $\gamma$, the map induced by $\tr(h)$ on homology sends the corresponding generator to this element $(i_b)_*(R(k_b))$.
  Equivalently, we want to know that the composite
  \[ S^n \xto{\id\sm \gamma_*} S^n \sm \sh{B_{\overline{f}}} \xto{\tr(h)}  S^n \sm \sh{E_f} \]
  picks out $(i_b)_*(R(k_b))$ in homology.
  But by \autoref{thm:tr-omega}, $\id \sm \gamma_*$ is the trace of $\omega_\gamma$, so by \autoref{thm:compose-traces2}, the above composite is homotopic to the trace of
  \[ {}_bB \odot B_p \xto{\omega_\gamma\odot \id} {}_bB \odot B_{\fbar} \odot B_p \xto{{}_bB \odot h} {}_bB \odot B_p \odot E_f
  \]
  But by \autoref{thm:fibcmpid2}, this is homotopic to the trace of
  \[ \xymatrix{
    \widecheck{S_{\fib{p}{b}}}\odot {}_{i_b} E \ar[r]^-{\widetilde{k_b} \odot \id} &
    \widecheck{S_{\fib{p}{b}}} \odot (\fib{p}{b})_{k_b} \odot {}_{i_b} E \ar[r]^-{\id \odot \ell} &
    \widecheck{S_{\fib{p}{b}}} \odot {}_{i_b} E  \odot E_f }
  \]
  and by \autoref{thm:compose-traces2} again, this is homotopic to the composite
  \[ S^n \xto{\tr(\widetilde{k_b})} \sh{(\fib{p}{b})_{k_b}} \xto{\tr (\ell)} \sh{E_f}. \]
  But by \autoref{thm:reidemeister_id}, ${\tr(\widetilde{k_b})}$ picks out $R(k_b)$ in homology, while \autoref{thm:inclusion} tells us that $\tr (\ell)$ acts by $(i_b)_*$ on homology.
  Thus, we obtain exactly $(i_b)_*(R(k_b))$, as desired.
\end{proof}

Finally, we can put everything together to prove \autoref{thm:reidemeister-intro}.

\begin{thm}\label{thm:reidemeister}
  Let $p\colon E\to B$ be a fibration such that $S_B$ is Costenoble-Waner dualizable, as is $S_{\fib{p}{b}}$ for each $b\in B$ (such as if $B$ and all fibers are closed smooth manifolds).
  Then the composite
  \begin{equation*}
    \bbZ \xto{R(\fbar)} \bbZ\sh{\pi_1 B_{\fbar}} \xto{\widehat{R_B}(f)} \bbZ\sh{\pi_1 E_f}
  \end{equation*}
  sends $1$ to $R(f)$.
\end{thm}
\begin{proof}
  We identify $\bbZ$ with $H_n(S^n)$, $\bbZ\;\sh{\pi_1 B_{\fbar}}$ with $H_n(S^n \sm \Lambda^{\fbar} B_+)$, and $\bbZ\;\sh{\pi_1 E_{f}}$ with $H_n(S^n \sm \Lambda^{f} E_+)$.
  By \autoref{thm:fibidentify2}, we can then identify $\widehat{R_B}(f)$ with $H_n(\tr(h))$, and by \autoref{thm:reidemeister_id} we can identify $R(\fbar)$ with $H_n(\tr(\hbar))$.
  Thus, by the functoriality of homology, it suffices to prove that the composite
  \[ S^n \xto{\tr(\hbar)} S^n \sm \Lambda^{\fbar} B_+ \xto{\tr(h)} S^n \sm \Lambda^f E_+ \]
  is homotopic to $\tr(\ftil)$.
  But this follows from \autoref{thm:compose-traces2} and \autoref{thm:hhcomp}.
\end{proof}

\section{Conclusions and future work}
\label{sec:conclusions}

As mentioned previously, while we have proven most theorems explicitly to improve the readability of this paper, most of them can be formulated and proven in the purely abstract context of~\cite{PS2,PS3, shulman:frbi}.
The only real topological input comes from \autoref{thm:mfd-cwdual}, which tells us what sort of objects the general theory of duality can be applied to, and \autoref{thm:reidemeister_id}, which identifies the general notion of trace with a more familiar numerical invariant.
Thus, analogous multiplicativity formulas will hold in any context where we can obtain theorems analogous to these two.
This includes, for instance, parametrized and equivariant homotopy theories.

Even discounting this potential for generalizations, however, we feel that the abstract point of view on multiplicativity presented here is likely to be valuable more widely.
As we saw in the introduction, it unifies the classical formula for the Lefschetz number of an orientable fibration with the various results in the literature on Nielsen numbers.
It also shows that orientability, far from being an essential ingredient, is just a condition which allows the general formula to be expressed in a simpler form.
(This is also of interest when considering generalizations to other contexts, where the relevant notion of ``orientability'' may be less obvious, but the abstract framework works just fine.)

\bibliographystyle{plain.bst}
\bibliography{traces_2}

\end{document}

%% file: decls.tex
\usepackage{amssymb,amsmath,stmaryrd,mathrsfs}
\def\definetac{\newif\iftac}    
\ifx\tactrue\undefined
  \definetac
  \ifx\state\undefined\tacfalse\else\tactrue\fi
\fi
\iftac\else\usepackage{amsthm}\fi
\usepackage[all,2cell]{xy}
\UseAllTwocells
\usepackage{enumitem}
\usepackage{color}
\definecolor{darkgreen}{rgb}{0,0.45,0} 
\usepackage[pagebackref,colorlinks,citecolor=darkgreen,linkcolor=darkgreen]{hyperref}
\usepackage{mathtools}          
\usepackage{graphics} 
\usepackage{braket}             

\usepackage{url}                
\usepackage{xspace}             

\makeatletter
\let\ea\expandafter

\def\mdef#1#2{\ea\ea\ea\gdef\ea\ea\noexpand#1\ea{\ea\ensuremath\ea{#2}\xspace}}
\def\alwaysmath#1{\ea\ea\ea\global\ea\ea\ea\let\ea\ea\csname your@#1\endcsname\csname #1\endcsname
  \ea\def\csname #1\endcsname{\ensuremath{\csname your@#1\endcsname}\xspace}}

\DeclareRobustCommand\widecheck[1]{{\mathpalette\@widecheck{#1}}}
\def\@widecheck#1#2{%
    \setbox\z@\hbox{\m@th$#1#2$}%
    \setbox\tw@\hbox{\m@th$#1%
       \widehat{%
          \vrule\@width\z@\@height\ht\z@
          \vrule\@height\z@\@width\wd\z@}$}%
    \dp\tw@-\ht\z@
    \@tempdima\ht\z@ \advance\@tempdima2\ht\tw@ \divide\@tempdima\thr@@
    \setbox\tw@\hbox{%
       \raise\@tempdima\hbox{\scalebox{1}[-1]{\lower\@tempdima\box
\tw@}}}%
    {\ooalign{\box\tw@ \cr \box\z@}}}

\newcount\foreachcount

\def\foreachletter#1#2#3{\foreachcount=#1
  \ea\loop\ea\ea\ea#3\@alph\foreachcount
  \advance\foreachcount by 1
  \ifnum\foreachcount<#2\repeat}

\def\foreachLetter#1#2#3{\foreachcount=#1
  \ea\loop\ea\ea\ea#3\@Alph\foreachcount
  \advance\foreachcount by 1
  \ifnum\foreachcount<#2\repeat}

\def\definescr#1{\ea\gdef\csname s#1\endcsname{\ensuremath{\mathscr{#1}}\xspace}}
\foreachLetter{1}{27}{\definescr}
\def\definecal#1{\ea\gdef\csname c#1\endcsname{\ensuremath{\mathcal{#1}}\xspace}}
\foreachLetter{1}{27}{\definecal}
\def\definebold#1{\ea\gdef\csname b#1\endcsname{\ensuremath{\mathbf{#1}}\xspace}}
\foreachLetter{1}{27}{\definebold}
\def\definebb#1{\ea\gdef\csname l#1\endcsname{\ensuremath{\mathbb{#1}}\xspace}}
\foreachLetter{1}{27}{\definebb}
\def\definefrak#1{\ea\gdef\csname f#1\endcsname{\ensuremath{\mathfrak{#1}}\xspace}}
\foreachletter{1}{9}{\definefrak} 
\foreachletter{10}{27}{\definefrak}
\def\definebar#1{\ea\gdef\csname #1bar\endcsname{\ensuremath{\overline{#1}}\xspace}}
\foreachLetter{1}{27}{\definebar}
\foreachletter{1}{8}{\definebar} 
\foreachletter{9}{15}{\definebar} 
\foreachletter{16}{27}{\definebar}
\def\definetil#1{\ea\gdef\csname #1til\endcsname{\ensuremath{\widetilde{#1}}\xspace}}
\foreachLetter{1}{27}{\definetil}
\foreachletter{1}{27}{\definetil}
\def\definehat#1{\ea\gdef\csname #1hat\endcsname{\ensuremath{\widehat{#1}}\xspace}}
\foreachLetter{1}{27}{\definehat}
\foreachletter{1}{27}{\definehat}
\def\definechk#1{\ea\gdef\csname #1chk\endcsname{\ensuremath{\widecheck{#1}}\xspace}}
\foreachLetter{1}{27}{\definechk}
\foreachletter{1}{27}{\definechk}
\def\defineul#1{\ea\gdef\csname u#1\endcsname{\ensuremath{\underline{#1}}\xspace}}
\foreachLetter{1}{27}{\defineul}
\foreachletter{1}{27}{\defineul}

\def\autofmt@n#1\autofmt@end{\mathrm{#1}}
\def\autofmt@b#1\autofmt@end{\mathbf{#1}}
\def\autofmt@l#1#2\autofmt@end{\mathbb{#1}\mathsf{#2}}
\def\autofmt@c#1#2\autofmt@end{\mathcal{#1}\mathit{#2}}
\def\autofmt@s#1#2\autofmt@end{\mathscr{#1}\mathit{#2}}
\def\autofmt@f#1\autofmt@end{\mathfrak{#1}}
\def\autofmt@u#1\autofmt@end{\underline{\smash{\mathsf{#1}}}}
\def\autofmt@U#1\autofmt@end{\underline{\underline{\smash{\mathsf{#1}}}}}
\def\autofmt@h#1\autofmt@end{\widehat{#1}}
\def\autofmt@r#1\autofmt@end{\overline{#1}}
\def\autofmt@t#1\autofmt@end{\widetilde{#1}}
\def\autofmt@k#1\autofmt@end{\check{#1}}

\def\auto@drop#1{}
\def\autodef#1{\ea\ea\ea\@autodef\ea\ea\ea#1\ea\auto@drop\string#1\autodef@end}
\def\@autodef#1#2#3\autodef@end{%
  \ea\def\ea#1\ea{\ea\ensuremath\ea{\csname autofmt@#2\endcsname#3\autofmt@end}\xspace}}
\def\autodefs@end{blarg!}
\def\autodefs#1{\@autodefs#1\autodefs@end}
\def\@autodefs#1{\ifx#1\autodefs@end%
  \def\autodefs@next{}%
  \else%
  \def\autodefs@next{\autodef#1\@autodefs}%
  \fi\autodefs@next}

\DeclareSymbolFont{bbold}{U}{bbold}{m}{n}
\DeclareSymbolFontAlphabet{\mathbbb}{bbold}

\mdef\delbar{\overline{\partial}}
\let\sm\wedge

\mdef\hf{\textstyle\frac12 }
\mdef\thrd{\textstyle\frac13 }
\mdef\qtr{\textstyle\frac14 }

\SelectTips{cm}{}
\newdir{ >}{{}*!/-10pt/@{>}}    

\mdef\Id{\mathrm{Id}}
\mdef\id{\mathrm{id}}
\alwaysmath{ell}
\alwaysmath{infty}
\alwaysmath{odot}
\def\frc#1/#2.{\frac{#1}{#2}}   
\mdef\ten{\mathrel{\otimes}}

\mdef\sqten{\mathrel{\boxtimes}}

\DeclareRobustCommand\widecheck[1]{{\mathpalette\@widecheck{#1}}}
\def\@widecheck#1#2{%
    \setbox\z@\hbox{\m@th$#1#2$}%
    \setbox\tw@\hbox{\m@th$#1%
       \widehat{%
          \vrule\@width\z@\@height\ht\z@
          \vrule\@height\z@\@width\wd\z@}$}%
    \dp\tw@-\ht\z@
    \@tempdima\ht\z@ \advance\@tempdima2\ht\tw@ \divide\@tempdima\thr@@
    \setbox\tw@\hbox{%
       \raise\@tempdima\hbox{\scalebox{1}[-1]{\lower\@tempdima\box
\tw@}}}%
    {\ooalign{\box\tw@ \cr \box\z@}}}

\DeclareMathOperator\Hom{Hom}

\newcommand{\too}[1][]{\ensuremath{\overset{#1}{\longrightarrow}}}

\let\into\hookrightarrow

\mdef\we{\overset{\sim}{\longrightarrow}}
\mdef\leftwe{\overset{\sim}{\longleftarrow}}

\let\xto\xrightarrow

\def\rightarrowtailfill@{\arrowfill@{\Yright\joinrel\relbar}\relbar\rightarrow}
\newcommand\xrightarrowtail[2][]{\ext@arrow 0055{\rightarrowtailfill@}{#1}{#2}}

\def\twoheadrightarrowfill@{\arrowfill@{\relbar\joinrel\relbar}\relbar\twoheadrightarrow}
\newcommand\xtwoheadrightarrow[2][]{\ext@arrow 0055{\twoheadrightarrowfill@}{#1}{#2}}

\def\slashedarrowfill@#1#2#3#4#5{%
  $\m@th\thickmuskip0mu\medmuskip\thickmuskip\thinmuskip\thickmuskip
   \relax#5#1\mkern-7mu%
   \cleaders\hbox{$#5\mkern-2mu#2\mkern-2mu$}\hfill
   \mathclap{#3}\mathclap{#2}%
   \cleaders\hbox{$#5\mkern-2mu#2\mkern-2mu$}\hfill
   \mkern-7mu#4$%
}
\def\rightslashedarrowfill@{%
  \slashedarrowfill@\relbar\relbar\mapstochar\rightarrow}
\newcommand\xslashedrightarrow[2][]{%
  \ext@arrow 0055{\rightslashedarrowfill@}{#1}{#2}}
\mdef\hto{\xslashedrightarrow{}}
\mdef\htoo{\xslashedrightarrow{\quad}}

\def\shvar#1#2{{\ensuremath{%
  \hspace{1mm}\makebox[-1mm]{$#1\langle$}\makebox[0mm]{$#1\langle$}\hspace{1mm}%
  {#2}%
  \makebox[1mm]{$#1\rangle$}\makebox[0mm]{$#1\rangle$}%
}}}
\def\sh{\shvar{}}

\long\def\my@drawfill#1#2;{%
\@skipfalse
\fill[#1,draw=none] #2;
\@skiptrue
\draw[#1,fill=none] #2;
}
\newif\if@skip
\newcommand{\skipit}[1]{\if@skip\else#1\fi}
\newcommand{\drawfill}[1][]{\my@drawfill{#1}}

\newif\ifhyperref
\@ifpackageloaded{hyperref}{\hyperreftrue}{\hyperreffalse}
\iftac
  \let\your@state\state
  \def\state#1{\gdef\currthmtype{#1}\your@state{#1}}
  \let\your@staterm\staterm
  \def\staterm#1{\gdef\currthmtype{#1}\your@staterm{#1}}
  \let\defthm\newtheorem
  \def\currthmtype{}
  \ifhyperref
    \def\autoref#1{\ref*{label@name@#1}~\ref{#1}}
  \else
    \def\autoref#1{\ref{label@name@#1}~\ref{#1}}
  \fi
  \AtBeginDocument{%
    \let\old@label\label%
    \def\label#1{%
      {\let\your@currentlabel\@currentlabel%
        \edef\@currentlabel{\currthmtype}%
        \old@label{label@name@#1}}%
      \old@label{#1}}
  }
\else
  \ifhyperref
    \def\defthm#1#2{%
      \newtheorem{#1}{#2}[section]%
      \expandafter\def\csname #1autorefname\endcsname{#2}%
      \expandafter\let\csname c@#1\endcsname\c@thm}
  \else
    \def\defthm#1#2{\newtheorem{#1}[thm]{#2}}
    \ifx\SK@label\undefined\let\SK@label\label\fi
    \let\old@label\label
    \let\your@thm\@thm
    \def\@thm#1#2#3{\gdef\currthmtype{#3}\your@thm{#1}{#2}{#3}}
    \def\currthmtype{}
    \def\label#1{{\let\your@currentlabel\@currentlabel\def\@currentlabel%
        {\currthmtype~\your@currentlabel}%
        \SK@label{#1@}}\old@label{#1}}
    \def\autoref#1{\ref{#1@}}
  \fi
\fi

\newtheorem{thm}{Theorem}[section]

\defthm{cor}{Corollary}
\defthm{prop}{Proposition}
\defthm{lem}{Lemma}
\defthm{sch}{Scholium}
\defthm{assume}{Assumption}
\defthm{claim}{Claim}
\defthm{conj}{Conjecture}
\defthm{hyp}{Hypothesis}
\defthm{fact}{Fact}
\iftac\theoremstyle{plain}\else\theoremstyle{definition}\fi
\defthm{defn}{Definition}
\defthm{notn}{Notation}
\iftac\theoremstyle{plain}\else\theoremstyle{remark}\fi
\defthm{rmk}{Remark}
\defthm{eg}{Example}
\defthm{egs}{Examples}
\defthm{ex}{Exercise}
\defthm{ceg}{Counterexample}

\def\thmqedhere{\expandafter\csname\csname @currenvir\endcsname @qed\endcsname}

\setitemize[1]{leftmargin=2em}
\setenumerate[1]{leftmargin=*}

\iftac
  \let\c@equation\c@subsection
\else
  \let\c@equation\c@thm
\fi
\numberwithin{equation}{section}

\@ifpackageloaded{mathtools}{\mathtoolsset{showonlyrefs,showmanualtags}}{}

\alwaysmath{alpha}
\alwaysmath{beta}
\alwaysmath{gamma}
\alwaysmath{Gamma}
\alwaysmath{delta}
\alwaysmath{Delta}
\alwaysmath{epsilon}
\mdef\ep{\varepsilon}
\alwaysmath{zeta}
\alwaysmath{eta}
\alwaysmath{theta}
\alwaysmath{Theta}
\alwaysmath{iota}
\alwaysmath{kappa}
\alwaysmath{lambda}
\alwaysmath{Lambda}
\alwaysmath{mu}
\alwaysmath{nu}
\alwaysmath{xi}
\alwaysmath{pi}
\alwaysmath{rho}
\alwaysmath{sigma}
\alwaysmath{Sigma}
\alwaysmath{tau}
\alwaysmath{upsilon}
\alwaysmath{Upsilon}
\alwaysmath{phi}
\alwaysmath{Pi}
\alwaysmath{Phi}
\mdef\ph{\varphi}
\alwaysmath{chi}
\alwaysmath{psi}
\alwaysmath{Psi}
\alwaysmath{omega}
\alwaysmath{Omega}

\makeatother
